\documentclass{amsart}
\usepackage[margin=1.00in]{geometry}

\usepackage{amsfonts}
\usepackage{setspace}
\usepackage{graphicx}

	\usepackage{mathrsfs}
\usepackage{hyperref}
\usepackage{graphicx}
\usepackage{amsmath, amsthm, amscd, amsfonts, amssymb, graphicx, color}
 \usepackage{url}
\usepackage{enumerate}
 \makeatletter
\let\reftagform@=\tagform@
\def\tagform@#1{\maketag@@@{(\ignorespaces\textcolor{blue}{#1}\unskip\@@italiccorr)}}
\renewcommand{\eqref}[1]{\textup{\reftagform@{\ref{#1}}}}
\makeatother
\usepackage{hyperref}
\hypersetup{colorlinks=true, linkcolor=red, anchorcolor=green,
citecolor=cyan, urlcolor=red, filecolor=magenta, pdftoolbar=true}

\usepackage{epsfig}        
\usepackage{epic,eepic}       

\newtheorem{theorem}{Theorem}
\theoremstyle{plain}

\newtheorem{corollary}{Corollary}

\newtheorem{example}{Example}

\newtheorem{lemma}{Lemma}

\newtheorem{remark}{Remark}

\numberwithin{equation}{section}

\DeclareMathOperator{\spe}{sp}

\def\etal{et al.\,}
\DeclareMathOperator{\tr}{tr}

\begin{document}
	\title[On the Davis--Wielandt radius inequalities]{On the Davis--Wielandt radius inequalities of Hilbert space operators}
	\author[M.W. Alomari ]{M.W. Alomari} 
	\address {Department of Mathematics, Faculty of Science and Information	Technology, Irbid National University,  P.O. Box 2600, Irbid, P.C. 21110, Jordan.}
	\email{mwomath@gmail.com}

	\date{\today}
	\subjclass[2010]{Primary: 47A30, 47A12; Secondary  47A63,
		47L05.}
	
	\keywords{Davis-Wielandt radius, Davis-Wielandt shell, $n\times n$ operator.}
	
\begin{abstract}
 In this work, some new upper and lower bounds of the Davis--Wielandt radius are introduced. Generalizations of some presented results  are obtained.   Some bounds of the Davis--Wielandt radius  for $n\times n$ operator matrices are established.   An extension of the Davis--Wielandt radius to the Euclidean operator radius is introduced.  
\end{abstract}

 	\maketitle

\section{Introduction} 
    Let $\mathscr{B}\left( \mathscr{H}\right) $ be the Banach algebra
  of all bounded linear operators defined on a complex Hilbert space
  $\left( \mathscr{H};\left\langle \cdot ,\cdot \right\rangle
  \right)$  with the identity operator  $1_\mathscr{H}$ in
  $\mathscr{B}\left( \mathscr{H}\right) $.    When $\mathscr{H} =
  \mathbb{C}^n$, we identify $\mathscr{B}\left( \mathscr{H}\right)$
  with the algebra $\mathfrak{M}_{n\times n}$ of $n$-by-$n$ complex
  matrices. Then, $\mathfrak{M}^{+}_{n\times n}$ is just the cone of
  $n$-by-$n$ positive semidefinite matrices.

  For a bounded linear operator $S$ on a Hilbert space
  $\mathscr{H}$, the numerical range $W\left(S\right)$ is the image
  of the unit sphere of $\mathscr{H}$ under the quadratic form $x\to
  \left\langle {Sx,x} \right\rangle$ associated with the operator.
  More precisely,
  \begin{align*}
  W\left( S \right) = \left\{ {\left\langle {Sx,x} \right\rangle :x
  	\in \mathscr{H},\left\| x \right\| = 1} \right\}.
  \end{align*}
  Also, the numerical radius is defined to be
  \begin{align*}
  w\left( S \right) = \sup \left\{ {\left| \lambda\right|:\lambda
  	\in W\left( S \right) } \right\} = \mathop {\sup }\limits_{\left\|
  	x \right\| = 1} \left| {\left\langle {Sx,x} \right\rangle }
  \right|.
  \end{align*}
  
  We recall that,  the usual operator norm of an operator $S$ is
  defined to be
  \begin{align*}
  \left\| S \right\|  = \sup \left\{ {\left\| {Sx} \right\|:x \in
  	\mathscr{H},\left\| x \right\| = 1} \right\}.
  \end{align*}

  One of the most interesting generalization of numerical range is the Davis--Wielandt shell; which is defined as
  \begin{align*}
DW\left( S \right) = \left\{ {\left( {\left\langle {Sx,x} \right\rangle ,\left\langle {Sx,Sx} \right\rangle } \right),x \in \mathscr{H},\left\| x \right\| = 1} \right\}
  \end{align*}
for any $S\in  \mathscr{B}\left(\mathscr{H}\right)$.  Clearly, the projection of the set $DW\left(S\right)$ on the first co-ordinate is $W\left(S\right)$.

The Davis--Wielandt shell and its radius were introduced and described firstly by Davis in \cite{Davis1} and \cite{Davis2} and Wielandt \cite{W}.
In fact, the Davis--Wielandt shell $DW\left( S \right)$ gives more information about the operator $S$ and $W\left( S \right)$. For instance, in the finite dimensional case, Li and Poon proved \cite{LP1} (see also \cite{LP2}) that the normal property of Hilbert space operators can be completely determined by
the geometrical shape of their Davis--Wielandt shells, namely, $S\in \mathfrak{M}_{n\times n}\left(\mathbb{C}\right)$ is normal if and
only if $DW\left(S\right)$ is a polyhedron in $\mathbb{C}\times \mathbb{R}$ identified with $\mathbb{R}^3$. Moreover, in finite dimensional case, the spectrum of an operator $S$; $\spe\left(S\right)$  is finite and $DW\left(S\right)$ is always closed, cf \cite[Theorem 2.3]{LP1}. 
These conditions are no longer equivalent for an infinite-dimensional operator $S$, cf \cite[Example 2.5]{LP1}.

In \cite{LSZ}, Lins \etal proved that, if $S\in  \mathfrak{M}_{n\times n}\left(\mathbb{C}\right)$ is normal with finite spectrum $\spe\left(S\right)$, then $DW\left(S\right)$ is the convex hull of the points $\left( {{\mathop{\rm Re}\nolimits} \left( {\lambda _j } \right),{\mathop{\rm Im}\nolimits} \left( {\lambda _j } \right),\left| {\lambda _j } \right|^2 } \right)$ $(j=1,\cdots,n)$,  for $\lambda _j\in \spe\left(S\right)$. Moreover,  
each point $\left( {{\mathop{\rm Re}\nolimits} \left( {\lambda _j } \right),{\mathop{\rm Im}\nolimits} \left( {\lambda _j } \right),\left| {\lambda _j } \right|^2 } \right)$ is an extreme point of $DW\left(S\right)$. In particular case, if $n=2$ i.e., $S=\left[ {\begin{array}{*{20}c}
a & b  \\
c & d  \\
\end{array}} \right]$ has eigenvalues $\lambda _1,\lambda _2$, then $DW\left(S\right)$ 
degenerates to the line segment joining the points $\left( {\lambda _1 ,\left| {\lambda _1 } \right|^2 } \right)$ and $\left( {\lambda _2 ,\left| {\lambda _2 } \right|^2 } \right)$. So that $\dim DW\left(S\right) \le 1$. In fact,
the condition $dim \left(DW\left(S\right)\right) \le 1$ holds if and only if $S$ is normal, with at most
two distinct eigenvalues. Otherwise, $DW\left(S\right)$  is an ellipsoid (without its interior) centered at 
$\left( {\frac{{\lambda _1  + \lambda _2 }}{2},\frac{1}{2}\tr\left( {\left| S \right|^2 } \right)} \right)$.
Also, it was proved that if $dim \left(DW\left(S\right)\right)\ge 2$, then $DW\left(S\right)$ is always convex. A complete description of $DW\left(S\right)$ for a quadratic
operator $S$ was given in see \cite{LP2}. For more details see also \cite{LPS1}, \cite{LPS2} and \cite{LSZ}.  \\

In \cite{W}, Wielandt shown that the Davis--Wielandt shell is a useful tool for characterizing the eigenvalues of matrices in the set
\begin{align*}
\{U^*TU + V^*SV : U, V \in \mathfrak{M}_{n\times n} \text{ are unitary}\}
\end{align*}
for given $T,S  \in \mathfrak{M}_{n\times n}$.

The Davis--Wielandt radius of $S\in \mathscr{B}\left(\mathscr{H}\right)$ is defined as
 \begin{align*}
dw\left( S \right) = \mathop {\sup }\limits_{\scriptstyle x \in H \hfill \atop 
	\scriptstyle \left\| x \right\| = 1 \hfill} \left\{ {\sqrt {\left| {\left\langle {Sx,x} \right\rangle } \right|^2  + \left\| {Sx} \right\|^4 } } \right\}.
\end{align*}
One can easily check that $dw\left( S \right)$ is unitarily invariant but it does not define a norm on $ \mathscr{B}\left(\mathscr{H}\right)$. However, it can be represented as a special case of the Euclidean operator radius as proved in Lemma \ref{lemma4}, (see below).

As a direct consequence, one can easily observe  that
\begin{align}
\label{eq1.1}\max \left\{ {w\left( S \right),\left\| S \right\|^2 } \right\} \le dw\left( T \right) \le \sqrt {w^2 \left( S \right) + \left\| S \right\|^4 } 
\end{align}
for all $S\in \mathscr{B}\left(\mathscr{H}\right)$. The inequalities are sharp.

The following result characterize the norm--parallelism of operators and an equality condition for the
 Davis--Wielandt radius \cite{ZMCN}.
\begin{theorem}
\label{thm1}Let $S\in \mathscr{B}\left(\mathscr{H}\right)$. Then the following conditions are equivalent:
\begin{enumerate}
\item $S \parallel 1_{\mathscr{H}}$.

\item $dw\left( S \right) = \sqrt {w^2 \left( S \right) + \left\| S \right\|^4 } $.
\end{enumerate}
\end{theorem} 
As a consequence of Theorem \ref{thm1}, we have the following result \cite{ZMCN}.
\begin{corollary}
\label{cor1} Let $S\in \mathscr{B}\left(\mathscr{H}\right)$. The following conditions are equivalent:
\begin{enumerate}
\item $dw\left( S \right) = \sqrt {w^2 \left( S \right) + \left\| S \right\|^4 } $.

\item $w\left( S \right) = \left\| S \right\|$.

\item $dw\left( S \right) = \left\| S\right\|\sqrt {1+ \left\| S \right\|^2 } $.

\item $S^*S \le w^2\left( S \right)  1_{\mathscr{H}}$.
\end{enumerate}
\end{corollary}
To see how much the lower bound in \eqref{eq1.1} is sharp, we note that, from \eqref{eq1.1} we have
\begin{align*}
w^2 \left( S \right) + \left\| S \right\|^4  \ge dw^2\left( S \right) \ge \max \left\{ {w^2\left( S \right),\left\| S \right\|^4 } \right\} &= \frac{{w^2 \left( S \right) + \left\| S \right\|^4 }}{2} + \frac{1}{2}\left| {w^2 \left( S \right) - \left\| S \right\|^4 } \right| 
\\
&\ge \frac{1}{2} \left({w^2 \left( S \right) + \left\| S \right\|^4 }\right),
\end{align*}
which is the Arithmetic mean of $w^2 \left( S \right)$ and  $\left\| S \right\|^4$, and this means that the lower bound in \eqref{eq1.1} is on a good distance from  $dw\left(\cdot\right)$.  \\

In their recent elegant work \cite{ZS}, Zamani and Shebrawi proved several inequalities involving the Davis--Wielandt radius and the numerical radii of Hilbert space operators. Among others, they shown that
\begin{align*}
dw\left( S \right) \le \sqrt[4]{{w\left( {\left| S \right|^4  + \left| S \right|^8 } \right) + 2w^2 \left( {\left| S \right|^2 S} \right)}}.
\end{align*}
Other interesting results were given in the same work \cite{ZS}, have been discussed and (in some cases) improved by Bhunia \etal \cite{BBBP}, among others, they shown that 
\begin{align*}
dw^2 \left( S \right) &\le \left\| {\left| S \right|^2  + \left| S \right|^4 } \right\|, 
\end{align*}
and
\begin{align*}
dw^2 \left( S \right) &\le \frac{1}{2}w\left( {S^2  + \left\| S \right\|^2 } \right) + \left\| {S^4 } \right\|. 
\end{align*}
An important property regarding the Davis--Wielandt radius of summand of two operators, was also  presented in \cite{BBBP}, as follows:
\begin{align}
dw\left(S_1+S_2\right)\le dw\left(S_1\right)+dw\left(S_2\right)+dw\left(S^*_1S_2+S_2^*S_1\right) \label{eq1.2}
\end{align}
for all $S_1,S_2\in\mathscr{B}\left(\mathscr{H}\right)$. Based on that, an upper bound for the Davis--Wielandt radius of $2\times 2$ off-diagonal operator matrix was given as follows:
\begin{align}
dw\left( {\left[ {\begin{array}{*{20}c}
		0 & A  \\
		B & 0  \\
		\end{array}} \right]} \right) \le \left\| A \right\|\sqrt {\frac{1}{4}  + \left\| A \right\|^2 }  +\left\| B \right\| \sqrt {\frac{1}{4}   + \left\| B \right\|^2 }.
\end{align}
For more dtails and new results concerning the Davis--Wielandt shell and the Davis--Wielandt radius of an operator, we refer the reader to \cite{BSP}, \cite{BP}, \cite{Feki}, \cite{LP1}--\cite{LSZ}.\\

This work is organized as follows: In the next section, a representation of an operator $S\in \mathscr{B}\left( \mathscr{H}\right)$ in terms of the Euclidean operator radius is given. Some new upper and lower bounds of the Davis--Wielandt radius are introduced.   Some examples verified that the presented results are better (in some cases) than \eqref{eq1.1} are also provided. In Section \ref{sec3}, some bounds of the Davis--Wielandt radius  for $n\times n$ operator matrices are established.  An extension of the Davis--Wielandt radius to the Euclidean operator radius is introduced. 
 
 \section{The Results} \label{sec2}
 In order to prove our results we need  a sequence of lemmas.
 \begin{lemma}
 	\label{lemma1}  
 	  The Power-Mean inequality reads
 		\begin{align}
 		a^\alpha  b^{1 - \alpha }  \le \alpha a + \left( {1 - \alpha } \right)b \le \left( {\alpha a^p  + \left( {1 - \alpha } \right)b^p } \right)^{\frac{1}{p}}    \label{eq2.1}
 		\end{align}
 		for all  $\alpha \in \left[0,1\right]$, $a,b\ge0$ and $ p\ge1$.
 		 
 \end{lemma}
 \begin{lemma}  
 	\label{lemma2}  Let  $A\in \mathscr{B}\left( \mathscr{H}\right)^+ $, then
 	\begin{align}
 	\left\langle {Ax,x} \right\rangle ^p  \le \left\langle {A^p x,x} \right\rangle, \qquad p\ge1  \label{eq2.2}
 	\end{align}
 	for any unit vector $x\in\mathscr{H}$. The inequality is reversed if $p\in \left[0,1\right]$.
 \end{lemma}
 The mixed Schwarz inequality was introduced in \cite{TK}, as
 follows:
 \begin{lemma}
 	\label{lemma3}  Let  $A\in \mathscr{B}\left( \mathscr{H}\right)  $, then
 	\begin{align}
 	\left| {\left\langle {Ax,y} \right\rangle} \right|  ^2  \le \left\langle {\left| A \right|^{2\alpha } x,x} \right\rangle \left\langle {\left| {A^* } \right|^{2\left( {1 - \alpha } \right)} y,y} \right\rangle, \qquad 0\le \alpha \le 1 \label{eq2.3}
 	\end{align}
 	for any   vectors $x,y\in \mathscr{H}$, where  $\left|A\right|=\left(A^*A\right)^{1/2}$.\\
 \end{lemma}
In some of our results we need the following two fundamental norm
estimates, which  are:
\begin{align}
\label{eq2.4}\left\| {A+ B } \right\| \le \frac{1}{2}\left(
{\left\| A \right\| + \left\| B \right\| + \sqrt
	{\left( {\left\| A \right\| - \left\| B \right\|} \right)^2  +
		4\left\| {A^{1/2} B^{1/2} } \right\|^2 } } \right),
\end{align}
and
\begin{align}
\label{eq2.5}\left\| {A^{1/2} B^{1/2} } \right\|  \le\left\| {A  B
} \right\| ^{1/2}.
\end{align}
Both estimates are valid for all positive operators $A,B \in \mathscr{B}\left( \mathscr{H}\right)$. Also, it should be noted that \eqref{eq2.4} is sharper than the triangle inequality as pointed out by Kittaneh in \cite{FK3}. \\

In order to establish our main first result concerning the the   Davis--Wielandt radius, we need to recall the concept of Euclidean operator radius of an $n$-tuple operator; which was introduced by 
 Popsecu in \cite{P}. Namely, for an $n$-tuple $ {\bf{T}} = \left( {T_1 , \cdots ,T_n } \right) \in  \mathscr{B}\left(\mathscr{H}\right)^{n}:=\mathscr{B}\left(\mathscr{H}\right)\times \cdots \times \mathscr{B}\left(\mathscr{H}\right)$; i.e.,  for $T_1,\cdots,T_n \in \mathscr{B}\left(\mathscr{H}\right)$.
The Euclidean operator radius of $T_1,\cdots,T_n $  is defined   by
\begin{align}
\label{eq2.6}w_{\rm{e}} \left( {T_1 , \cdots ,T_n } \right): = \mathop {\sup }\limits_{\left\| x \right\| = 1} \left( {\sum\limits_{i = 1}^n {\left| {\left\langle {T_i x,x} \right\rangle } \right|^2 } } \right)^{1/2}\qquad\text{for all}\,\, x\in\mathscr{H}. \\\nonumber
\end{align}

The following properties of the Euclidean operator radius were proved in  \cite{P}, \cite{MKS} and \cite{SMS}:\\
\begin{enumerate}
	\item  $w_{\rm{e}} \left( {T_1 , \cdots ,T_n } \right)=0$ if and only if $T_k=0$ for each $k=1,\cdots,n$. \\
	
	\item  $w_{\rm{e}} \left( {\lambda T_1 , \cdots ,\lambda T_n } \right)= \left|\lambda\right| w_{\rm{e}} \left( {T_1 , \cdots ,T_n } \right)$. \\
	
	\item $w_{\rm{e}} \left( {A_1  + B_1 , \cdots ,A_n  + B_n } \right) \le w_{\rm{e}} \left( {A_1 , \cdots ,A_n } \right) + w_{\rm{e}} \left( {B_1 , \cdots ,B_n } \right)$.\\
	
	\item  $w_{\rm{e}} \left( {X^*T_1X , \cdots ,X^*T_nX } \right)= \left\|X\right\| w_{\rm{e}} \left( {T_1 , \cdots ,T_n } \right)$. \\

	\item $w_{\rm{e}} \left( {T_1 , \cdots ,T_n } \right)=w_{\rm{e}}\left( {T^*_1 , \cdots ,T^*_n } \right)$.\\
	
	\item $w_{\rm{e}} \left( {T^*_1T_1 , \cdots ,T^*_nT_n } \right)=w_{\rm{e}} \left( {T_1T^*_1 , \cdots ,T_nT^*_n } \right)$.\\
\end{enumerate}
for every $T_k,A_k,B_k,X\in \mathscr{B}\left(\mathscr{H}\right)$ $(1\le k \le n)$ and every scalar $\lambda\in \mathbb{C}$. \\

The Euclidean operator radius was generalized in \cite{MKS} as follows: 
\begin{align*} 
w_{p} \left( {T_1 , \cdots ,T_n } \right): = \mathop {\sup }\limits_{\left\| x \right\| = 1} \left( {\sum\limits_{i = 1}^n {\left| {\left\langle {T_i x,x} \right\rangle } \right|^p } } \right)^{1/p}, \qquad  p\ge1.
\end{align*} 
Clearly, for $p=2$ we refer to the Euclidean operator radius $w_{\rm{e}} \left( {\cdot,\ldots,\cdot} \right)$.\\

The following relation between the Euclidean operator radius $w_{\rm{e}} \left( {S , S^*S} \right)$ and  the   Davis--Wielandt radius 
$ dw\left(S\right)$ holds for every $S\in  \mathscr{B}\left(\mathscr{H}\right)$.
 \begin{lemma}
\label{lemma4} Let $S\in  \mathscr{B}\left(\mathscr{H}\right)$. Then
  \begin{align}
\label{eq2.7}w_{\rm{e}} \left( {S , S^*S} \right)= dw\left(S\right)
 \end{align}
 \end{lemma}
\begin{proof}
Setting $n=2$, $T_1=S$ and $T_2=S^*S$\,\,\,  $\left(S\in \mathscr{B}\left(\mathscr{H}\right)\right)$ in \eqref{eq2.6}, we have
\begin{align*}
w_{\rm{e}} \left( {S , S^*S} \right) := \mathop {\sup }\limits_{\left\| x \right\| = 1} \left( { \left| {\left\langle {S x,x} \right\rangle } \right|^2   +\left| {\left\langle {S^*S x,x} \right\rangle } \right|^2 } \right)^{1/2}
&= \mathop {\sup }\limits_{\left\| x \right\| = 1} \left\{ {\sqrt {\left| {\left\langle {Sx,x} \right\rangle } \right|^2  + \left\| {Sx} \right\|^4 } } \right\}
\\
&= dw\left(S\right),
\end{align*}
which gives the Davis--Wielandt radius of $S$, as required.
\end{proof}
 \begin{theorem}
 	\label{thm2}Let $S\in \mathscr{B}\left(\mathscr{H}\right)$. Then $dw\left(S\right)=  \sqrt{2} w\left(S\right)$  if and only if $S$ is selfadjoint idempotent operator.
 \end{theorem}
 \begin{proof}
 To prove the `only if part', from Lemma \ref{lemma4}, we have $w_{\rm{e}} \left( {S , S^*S} \right)= dw\left(S\right)$ for any $S\in \mathscr{B}\left(\mathscr{H}\right)$. Clearly if $S$ is selfadjoint idempotent operator, then  $dw\left(S\right)=w_{\rm{e}} \left( {S , S^*S} \right)=w_{\rm{e}} \left( {S , S^2} \right)=w_{\rm{e}} \left( {S , S} \right)$. On the other hand, by setting $n=2$ and $T_1=T_2=S$, in \eqref{eq2.6}, we get 
 	$w_{\rm{e}} \left( {S , S} \right)= \sqrt{2} w\left(S\right)$. Hence $dw\left(S\right)=  \sqrt{2} w\left(S\right)$. The `if part' follows by noting that, $S^*S=S^2$ if and only if $S$ is selfadjoint and therefore $S^*S=S$, when $S$ is an idempotent operator, i.e., $S^2=S$.
 \end{proof}

 It's well-known that if $S\in \mathscr{B}\left(\mathscr{H}\right)$ is selfadjoint operator, then $\left\|S\right\|=w\left(S\right)$. Thus, according to Theorem \ref{thm2}, the equality  $\left\|S\right\|=w\left(S\right)=\frac{1}{\sqrt{2}}dw\left(S\right)$, holds when $S$ is selfadjoint idempotent operator. For example, take $ S = \left[ {\begin{array}{*{20}c}
 	1 & 0  \\
 	0 & 0  \\
 	\end{array}} \right]$. Therefore, we have $\left\|S\right\|=w\left(S\right)=1=\frac{1}{\sqrt{2}}dw\left(S\right)$. 
 
 It's convenient to note that,  Kittaneh \cite{K1} proved that if $S^2=0$, then $w\left(S\right)=\frac{1}{2}\left\|S\right\|$ for all $S\in \mathscr{B}\left(\mathscr{H}\right)$ with $\dim\left(\mathscr{H}\right)\le 2$. But it is not possible to have $dw\left(S\right)=\sqrt{2}w\left(S\right)=\frac{\sqrt{2}}{2}\left\|S\right\|$, because the first equality holds when $S$ is selfadjoint idempotent operator, which in turn implies that $w\left(S\right)=\left\|S\right\|$; hence, we have $dw\left(S\right)=\sqrt{2}\left\|S\right\|=\frac{\sqrt{2}}{2}\left\|S\right\|$, which of course impossible. Furthermore, one can show if $S\in \mathscr{B}\left(\mathscr{H}\right)$ such that $S^2=0$, then $S$ is not selfadjoint operator; except the zero operator.

In 2005, Kittaneh \cite{K2} proved that
\begin{align}
\frac{1}{4}\|S^*S+SS^*\|\le  w^2\left(S\right) \le \frac{1}{2}\|S^*S+SS^*\|\label{eq2.8}    
\end{align}
for Hilbert space operator $S\in \mathscr{B}\left(\mathscr{H}\right)$.  These inequalities were also reformulated and generalized in \cite{EF} but in terms of the Cartesian decomposition.\\

The following result extends \eqref{eq2.8} for the generalized Euclidean operator radius.   
\begin{lemma}
\label{lemma5}	Let $T_k\in \mathscr{B}\left(\mathscr{H}\right)$ $(k=1,\cdots,n)$. Then	
	\begin{align}
	\label{eq2.9} \frac{1}{{2^{p+1}n^{p - 1} }}\left\| {\sum\limits_{k = 1}^n {T_k^* T_k  + T_k T_k^* } } \right\|^p \le w^{2p}_{2p} \left( {T_1 , \cdots ,T_n } \right) \le \frac{1}{{2^p }}\left\| {\sum\limits_{k = 1}^n {\left( {T_k^* T_k  + T_k T_k^* } \right)^p } } \right\|
	\end{align}
	for all $p\ge1$.
\end{lemma}
\begin{proof}
	Let $P_k+iQ_k$ be the Cartesian decomposition of $T_k$ for all $k=1,\cdots,n$.
	As in the proof of \eqref{eq2.8} in \cite{K2}, we have
	\begin{align*}
	\left| {\left\langle {T_k x,x} \right\rangle } \right|^{2p}   = \left( {\left\langle {P_k x,x} \right\rangle ^2  + \left\langle {Q_k x,x} \right\rangle ^2 } \right)^p    
	 \ge \frac{1}{{2^p }}\left( {\left| {\left\langle {P_k x,x} \right\rangle } \right| + \left| {\left\langle {Q_k x,x} \right\rangle } \right|} \right)^{2p}    
	&\ge \frac{1}{{2^p }}\left| {\left\langle {P_k x,x} \right\rangle  + \left\langle {Q_k x,x} \right\rangle } \right|^{2p}  \\ 
	&= \frac{1}{{2^p }}\left| {\left\langle {P_k  \pm Q_k x,x} \right\rangle } \right|^{2p}.    
	\end{align*} 
	Summing over $k$ and then taking the supremum over all unit vector $x\in \mathscr{H}$, we get
	\begin{align*}
	w^{2p}_{2p} \left( {T_1 , \cdots ,T_n } \right)=\mathop {\sup }\limits_{\left\| x \right\| = 1} \sum\limits_{k = 1}^n {\left| {\left\langle {T_k x,x} \right\rangle } \right|^{2p}} 
	 &\ge \frac{1}{{2^p }}\mathop {\sup }\limits_{\left\| x \right\| = 1} \sum\limits_{k = 1}^n {\left| {\left\langle {P_k  \pm Q_k x,x} \right\rangle } \right|^{2p} }  
	\\
	&\ge \frac{1}{{2^p }}\frac{1}{{n^{p - 1} }}\mathop {\sup }\limits_{\left\| x \right\| = 1} \left( {\sum\limits_{k = 1}^n {\left| {\left\langle {P_k  \pm Q_k x,x} \right\rangle } \right|^2 } } \right)^p  \\ 
	&= \frac{1}{{2^p }}\frac{1}{{n^{p - 1} }}\left\| {\sum\limits_{k = 1}^n {\left( {P_k  \pm Q_k } \right)^2 } } \right\|^p,  
	\end{align*}
	where we have used the Jensen's inequality in the last inequality. Thus,
	\begin{align*}
	2w^{2p}_{2p} \left( {T_1 , \cdots ,T_n } \right) &\ge \frac{1}{{2^p }}\frac{1}{{n^{p - 1} }}\left\| {\sum\limits_{k = 1}^n {\left( {P_k  + Q_k } \right)^2 } } \right\|^p  + \frac{1}{{2^p }}\frac{1}{{n^{p - 1} }}\left\| {\sum\limits_{k = 1}^n {\left( {P_k  - Q_k } \right)^2 } } \right\|^p  \\ 
	&\ge \frac{1}{{2^p }}\frac{1}{{n^{p - 1} }}\left\| {\sum\limits_{k = 1}^n {\left( {P_k  + Q_k } \right)^2 }  + \sum\limits_{k = 1}^n {\left( {P_k  - Q_k } \right)^2 } } \right\|^p  \\ 
	&= \frac{1}{{2^p }}\frac{1}{{n^{p - 1} }}\left\| {\sum\limits_{k = 1}^n {\left\{ {\left( {P_k  + Q_k } \right)^2  + \left( {P_k  - Q_k } \right)^2 } \right\}} } \right\|^p  \\ 
	&=  \frac{1}{{n^{p - 1} }}\left\| {\sum\limits_{k = 1}^n {P_k^2  + Q_k^2 } } \right\|^p \\  
	 &=     \frac{1}{{n^{p - 1} }}\left\| {\sum\limits_{k = 1}^n {\frac{T_k^* T_k  + T_k T_k^*}{2} } } \right\|^p\\
	 &=   \frac{1}{{2^pn^{p - 1} }}\left\| {\sum\limits_{k = 1}^n { T_k^* T_k  + T_k T_k^*  } } \right\|^p,  
	\end{align*}
	and hence,
	\begin{align*}
	w^{2p}_{2p} \left( {T_1 , \cdots ,T_n } \right)   \ge    \frac{1}{{2^{p+1}n^{p - 1} }}\left\| {\sum\limits_{k = 1}^n {T_k^* T_k  + T_k T_k^* } } \right\|^p,  
	\end{align*}
	which proves the left hand side of the inequality in \eqref{eq2.9}.

	To prove the second inequality, for every unit vector $x\in \mathscr{H}$ we have
	\begin{align*}
	\sum\limits_{k = 1}^n {\left| {\left\langle {T_k x,x} \right\rangle } \right|^{2p} }   = \sum\limits_{k = 1}^n {\left( {\left\langle {P_k x,x} \right\rangle ^2  + \left\langle {Q_k x,x} \right\rangle ^2 } \right)^p }    
	&\le \sum\limits_{k = 1}^n {\left( {\left\langle {P_k^2 x,x} \right\rangle  + \left\langle {Q_k^2 x,x} \right\rangle } \right)^p }  \\ 
	&= \sum\limits_{k = 1}^n {\left\langle {\left( {P_k^2  + Q_k^2 } \right)x,x} \right\rangle ^p },
	\end{align*}
	which implies that 
	\begin{align*} 
	\mathop {\sup }\limits_{\left\| x \right\| = 1}\sum\limits_{k = 1}^n {\left| {\left\langle {T_k x,x} \right\rangle } \right|^{2p} } =w_{2p}^{2p} \left( {T_1 , \cdots ,T_1 } \right)&\le \mathop {\sup }\limits_{\left\| x \right\| = 1} \sum\limits_{k = 1}^n {\left\langle {\left( {P_k^2  + Q_k^2 } \right)x,x} \right\rangle ^p }  
	\\
	&= \left\| {\sum\limits_{k = 1}^n {\left( {P_k^2  + Q_k^2 } \right)^p } } \right\| 
 = \frac{1}{{2^p }}\left\| {\sum\limits_{k = 1}^n {\left( {T_k^* T_k  + T_k T_k^* } \right)^p } } \right\|,
	\end{align*}
	which proves the right hand side of \eqref{eq2.9}.
\end{proof}
\begin{remark}
	In particular, setting $n=2$ and $p=1$ in \eqref{eq2.9} we get
	\begin{align*}
	\frac{1}{4} \left\| {  {T_1^* T_1  + T_1 T_1^* +T_2^* T_2 + T_2 T_2^* } } \right\|  &\le w^2_{\rm{e}} \left( {T_1,T_2} \right) 
	\\
	&\le \frac{1}{{2  }}\left\| {T_1^* T_1  + T_1 T_1^* +T_2^* T_2 + T_2 T_2^*   } \right\|.
	\end{align*}
	Moreover, if we choose $T_1=T_2=T$, then
	\begin{align*}
	\frac{1}{2} \left\| {  {T^* T  + T T^*  } } \right\|   \le w^2_{\rm{e}} \left( {T,T} \right) 
	\le  \left\| {T^* T  + T T^*} \right\|.
	\end{align*}
	but $w_{\rm{e}} \left( {T,T} \right)=\sqrt{2} w\left(T\right)$, thus the last inequality above reduces to the Kittaneh inequality \eqref{eq2.8}.
\end{remark}
Now, based on Lemmas \ref{lemma4} and \ref{lemma5}, we can introduce our first main result, as follows:
\begin{theorem}
\label{thm3}	Let $S\in  \mathscr{B}\left(\mathscr{H}\right)$. Then
\begin{align}
\label{eq2.10}\frac{1}{4} \left\| {  {\left|S\right|^2 +\left|S^*\right|^2 + 2\left|S\right|^4   } } \right\|   \le  dw^2\left(S\right) 
\le \frac{1}{{2  }}\left\| {\left|S\right|^2 +\left|S^*\right|^2 +2\left|S\right|^4     } \right\|.
\end{align} 
\end{theorem}

\begin{proof}
Setting $n=2$, $p=1$, $T_1=A$ and $T_2=B$ in \eqref{eq2.9}, we get
\begin{align*}
\frac{1}{4} \left\| {  {A^* A  + A A^* +B^* B + B B^* } } \right\|  &\le w^2_{\rm{e}} \left( {A,B} \right) 
\\
&\le \frac{1}{{2  }}\left\| {A^* A  + A A^* + B^*B + B B^*    } \right\|.
\end{align*}
Replacing $A$ by $S$ and $B$ by $S^*S$, we get
\begin{align*}
\frac{1}{4} \left\| {  {S^* S  +SS^* + 2\left|S\right|^4   } } \right\|  \le w^2_{\rm{e}} \left( {S,S^*S} \right) 
\le \frac{1}{{2  }}\left\| {S^* S  + S S^* +2\left|S\right|^4     } \right\|.
\end{align*} 
But as we have shown   in Lemma \ref{lemma4} that , $w_{\rm{e}} \left( {S , S^*S} \right)= dw\left(S\right)$, hence we have
\begin{align*}
\frac{1}{4} \left\| {  {\left|S\right|^2 +\left|S^*\right|^2 + 2\left|S\right|^4   } } \right\|   \le  dw^2\left(S\right) 
\le \frac{1}{{2  }}\left\| {\left|S\right|^2 +\left|S^*\right|^2 +2\left|S\right|^4     } \right\|,
\end{align*} 
as desired.
\end{proof}

To see that the second inequality in \eqref{eq2.10} is a refinement of the second inequality in \eqref{eq1.1}, assume  
 $SS^*\le  S^*S\le w^2\left(S\right)1_{\mathscr{H}}$. Thus, from \eqref{eq2.10} we have
\begin{align*}
dw^2\left(S\right)  \le \frac{1}{2} \left\| {  {\left|S\right|^2 +\left|S^*\right|^2 + 2\left|S\right|^4   } } \right\| &\le \frac{1}{2} \left\| {  {w^2\left(S\right)1_{\mathscr{H}} +w^2\left(S\right)1_{\mathscr{H}}+ 2w^4\left(S\right)1_{\mathscr{H}}   } } \right\|  
\\
&\le      w^2  \left(S\right)+\left\|S\right\|^4,   
\end{align*}
follows  by assumption, since $w\left(S\right)=\left\|S\right\|$ (see Corollary \ref{cor1}), which implies that
\begin{align*}
dw \left(S\right)  \le \sqrt{ \frac{1}{2} \left\| {  {\left|S\right|^2 +\left|S^*\right|^2 + 2\left|S\right|^4   } } \right\|} \le \sqrt{w^2  \left(S\right)+\left\|S\right\|^4 }= \left\|S\right\|    \sqrt{1  +  \left\|S\right\|^2 },
\end{align*}
which means that the right-hand side of \eqref{eq2.10} refines the right-hand side of \eqref{eq1.1}.\\

\begin{example}
\label{example1}	Let $S = \left[ {\begin{array}{*{20}c}
		0 & 1  \\
		2& 1  \\
		\end{array}} \right]$. We have  $\|S\|=2.28825$ and $w\left(S\right)=2.08114$. The upper bound of \eqref{eq1.1} gives $dw\left(S\right) \le  5.63449$. However, by applying \eqref{eq2.10}, we have $dw\left(S\right) \le 5.61938$, which implies that, the upper bound in \eqref{eq2.10} is better than the upper bound in \eqref{eq1.1}. 
\end{example}
 
 \begin{remark}
Following the same approach considred in the proof of Theorem \ref{thm3}, another interesting inequalities could be deduced from the obtained inequalities in   \cite{P}, \cite{MKS}, and \cite{SMS}.
 \end{remark}

The following result refines sharply the upper bound in \eqref{eq1.1}. 
 \begin{theorem}
 \label{thm4}	If $S\in \mathscr{B}\left(\mathscr{H}\right)$, then
 	\begin{align}
 \label{eq2.11} \frac{1}{{\sqrt 2 }}\left\|S+S^* S\right\|\le  dw\left( S \right)   \le  \sqrt{\left\|
   	{  \frac{1}{4}\left( \left| S \right|   + \left| {S^* } \right|  \right)^2 +  \left| {  S} \right|^4  }\right\|}
   \le   \sqrt{\frac{1}{4}\left( {  \left\|  S \right\|+\left\|  S^2 \right\|^{1/2} }\right)^2  +  \left\| S \right\|^4}.  
   	\end{align}	
 \end{theorem}
  
 \begin{proof}
 Since we have
 \begin{align*}
 dw^2\left(S\right)=\mathop {\sup }\limits_{\scriptstyle x \in \mathscr{H} \hfill \atop 
 	\scriptstyle \left\| x \right\| = 1 \hfill}\left\{\left| {\left\langle {Sx,x} \right\rangle } \right|^2  + \left| {\left\langle {S^* Sx,x} \right\rangle } \right|^2  \right\} 
 &\ge\frac{1}{2}\mathop {\sup }\limits_{\scriptstyle x \in \mathscr{H} \hfill \atop 
 	\scriptstyle \left\| x \right\| = 1 \hfill}\left\{\left| {\left\langle {Sx,x} \right\rangle } \right|   + \left| {\left\langle {S^* Sx,x} \right\rangle } \right|   \right\}^2
 \\ 
 &=\frac{1}{2}\mathop {\sup }\limits_{\scriptstyle x \in \mathscr{H} \hfill \atop 
 	\scriptstyle \left\| x \right\| = 1 \hfill}\left\{\left| {\left\langle {Sx,x} \right\rangle  +\left\langle {S^* Sx,x} \right\rangle } \right|   \right\}^2
 \\ 
 &= \frac{1}{2}\mathop {\sup }\limits_{\scriptstyle x \in \mathscr{H} \hfill \atop 
 	\scriptstyle \left\| x \right\| = 1 \hfill}\left\{\left| {\left\langle {\left(S+S^* S\right)x,x} \right\rangle  } \right|   \right\}^2
  =\frac{1}{2}\left\|S+S^*S\right\|^2,
 \end{align*}
 which proves the first inequality in \eqref{eq2.11}. Also, since we have
 	\begin{align*}
 	dw^2 \left(S\right) &=\mathop {\sup }\limits_{\scriptstyle x \in \mathscr{H} \hfill \atop 
 		\scriptstyle \left\| x \right\| = 1 \hfill} \left\{
 \left| {\left\langle {Sx,x} \right\rangle } \right|^2  + \left\| {Sx} \right\|^4\right\} 
 \\
  &= \mathop {\sup }\limits_{\scriptstyle x \in \mathscr{H} \hfill \atop 
 	\scriptstyle \left\| x \right\| = 1 \hfill}\left\{\left| {\left\langle {Sx,x} \right\rangle } \right|^2  + \left| {\left\langle {S^* Sx,x} \right\rangle } \right|^2  \right\}
 \\ 
&\le \mathop {\sup }\limits_{\scriptstyle x \in \mathscr{H} \hfill \atop 
	\scriptstyle \left\| x \right\| = 1 \hfill} \left\{\left(\left\langle {\left| S \right| x,x} \right\rangle ^{\frac{1}{2}} \left\langle {\left| {S^* } \right|  x,x} \right\rangle ^{\frac{1}{2}}  \right)^2+ \left\langle {\left| {S^* S} \right|^2 x,x} \right\rangle   \right\}\qquad  (\text{by \eqref{eq2.3}})
\\ 
&\le \mathop {\sup }\limits_{\scriptstyle x \in \mathscr{H} \hfill \atop 
	\scriptstyle \left\| x \right\| = 1 \hfill}\left[ {\left\langle {\frac{\left| S \right| +\left|S^* \right| }{2} x,x} \right\rangle^2    + \left\langle {\left| {S^*S} \right|^2  x,x} \right\rangle } \right] \qquad  (\text{by \eqref{eq2.1}})
\\ 
&\le \mathop {\sup }\limits_{\scriptstyle x \in \mathscr{H} \hfill \atop 
	\scriptstyle \left\| x \right\| = 1 \hfill}\left[ {\left\langle {\left(\frac{\left| S \right|   + \left| {S^* } \right| }{2} \right)^2 x,x} \right\rangle     + \left\langle {\left| {S^* S} \right|^2  x,x} \right\rangle } \right]
\\ 
&= \mathop {\sup }\limits_{\scriptstyle x \in \mathscr{H} \hfill \atop 
	\scriptstyle \left\| x \right\| = 1 \hfill}\left\langle {\left( { \left(\frac{\left| S \right|   + \left| {S^* } \right| }{2} \right)^2 + \left| {S^* S} \right|^2 } \right)x,x} \right\rangle
 =\frac{1}{4}\left\|
	{  \left( \left| S \right|   + \left| {S^* } \right|  \right)^2 + 4\left| {S^* S} \right|^2  }\right\|,
 	\end{align*}
and this proves the second inequality in \eqref{eq2.11}. Applying the triangle inequality on the above inequality, we get
 \begin{align*}
	dw^2 \left(S\right) \le \frac{1}{4} \left\|
 {  \left( \left| S \right|   + \left| {S^* } \right|  \right)^2 + 4\left| {S^* S} \right|^2  }\right\|
  &\le \frac{1}{4} \left\| {  \left( \left| S\right|   + \left| {S^* } \right|  \right)^2 }\right\| + \left\| \left| {S^* S} \right|^2  \right\|
   \\
   &= \frac{1}{4} \left\| {    \left| S \right|   + \left| {S^* } \right|  }\right\|^2  + \left\| \left| { S} \right|^4  \right\|.
 \end{align*}
Now, applying \eqref{eq2.4} to the first term in the above inequality, we get $\left\|\left| S \right| + \left| {S^* } \right|  \right\|\le \left\|  S \right\|+\left\|  S^2 \right\|^{1/2} $. Now substituting this inequality in the last inequality above, we get the third inequality in \eqref{eq2.11}, and this completes the proof.
\end{proof}

\begin{remark}
	\label{rem3}We note that, a refinement of the inequality \eqref{eq2.11} could be stated as follows:
	\begin{align*}
	\frac{1}{{\sqrt 2 }}\left\|S+S^* S\right\|\le dw\left(S\right)  \le  \sqrt{w\left({  \frac{1}{4}\left( \left|S\right|   + \left| {S^* } \right|  \right)^2 +  \left| {S} \right|^4  }\right)}.
	\end{align*}	
	Consider $S$ as in Example \ref{example1}. Applying the above inequality, we get  $dw\left(S\right)\le 5.59709$, which is better than the result obtained by \eqref{eq2.10}. Furthermore, \eqref{eq2.8} gives that
	\begin{align*}
	dw\left(S\right)  \le  \sqrt{w\left({  \frac{1}{4}\left( \left|S\right|   + \left| {S^* } \right|  \right)^2 +  \left| {S} \right|^4  }\right)} \le \sqrt[4]{\frac{1}{2} \left\|T^*T+TT^* \right\|},
	\end{align*}
	where $T=\frac{1}{4}\left( \left|S\right|   + \left| {S^* } \right|  \right)^2 +  \left| {S} \right|^4$. 
	Employing the previous second upper bound for $S$ in Example \ref{example2},  we get the same result as those obtained by \eqref{eq2.11} and \eqref{eq1.1}, even we use \eqref{eq2.8}; which indeed refines \eqref{eq2.11}. 
\end{remark}

\begin{example}
\label{example2}Let $S = \left[ {\begin{array}{*{20}c}
	1 & 1 & 0  \\
	0 & 1 & 1  \\
	1 & 0 & 1  \\
	\end{array}} \right]$. We have $w\left(S\right)= \|S\|=2$. Employing the sharp lower bound in \eqref{eq1.1} we get that $dw\left(S\right)\ge 4$. By applying the lower bound in \eqref{eq2.11}, we get  $dw\left(S\right)\ge 3 \sqrt{2}= 4.2426$, which means that the lower bound in \eqref{eq2.11} is better than that one given in \eqref{eq1.1}.  

Also, applying the first upper bound in \eqref{eq2.11} we have
$  dw\left(S\right) \le 2\sqrt{5}=4.47214$, which gives the same result if one chooses to apply the upper bound in \eqref{eq1.1}. 
\end{example}

\begin{remark}
In \cite{K1}, Kittaneh proved that if $S\in \mathscr{B}\left(\mathscr{H}\right)$ is such that $S^2=0$, then $w\left(S\right)=\frac{1}{2}\left\|S\right\|$. Under this assumption, the inequality \eqref{eq1.1} becomes
\begin{align*}
\max \left\{ {\frac{1}{2}\left\|S\right\|,\left\| S \right\|^2 } \right\} \le dw\left( S \right) \le \sqrt {\frac{1}{4}\left\|S\right\|^2 + \left\| S \right\|^4 }. 
\end{align*}
Similarly, the (second) upper bound in  \eqref{eq2.11}  reduced to the form
\begin{align*}
 dw\left( T \right)  \le \sqrt{\left\|
	{  \frac{1}{4} \left( \left| S \right|   + \left| {S^* } \right|  \right)^2 +  \left| { S} \right|^4  }\right\|}
\le    \sqrt{   \frac{1}{4}\left\|  S \right\|^2 +  \left\|S  \right\|^4}.
\end{align*}	
 \end{remark}

A generalization of the upper bound in Theorem \ref{thm3} is considered as follows:
\begin{theorem}
\label{thm5}	Let    $S\in \mathscr{B}  \left(\mathscr{H}\right)$,   $0 \le \alpha \le 1$ and $r\ge2$. Then
	\begin{align}
  	dw^r\left(S\right) \le  \frac{2^{\frac{r}{2}}}{4}  \left\|\left| S \right|^{2r\alpha} +\left| {S^* } \right|^{2r\left(1-\alpha\right)}+\left| S^*S \right|^{2r\alpha} +\left| {S^*S } \right|^{2r\left(1-\alpha\right)}\right\|.\label{eq2.12}
	\end{align}
\end{theorem}
\begin{proof}
	Let $x\in \mathscr{H}$ be  unit vector, then 
		\begin{align*}
		dw^2 \left(S\right)  =\mathop {\sup }\limits_{\scriptstyle x \in \mathscr{H} \hfill \atop 
		\scriptstyle \left\| x \right\| = 1 \hfill} \left\{
	\left| {\left\langle {Sx,x} \right\rangle } \right|^2  + \left\| {Sx} \right\|^4\right\} 
	 = \mathop {\sup }\limits_{\scriptstyle x \in \mathscr{H} \hfill \atop 
		\scriptstyle \left\| x \right\| = 1 \hfill}\left\{\left| {\left\langle {Sx,x} \right\rangle } \right|^2  + \left| {\left\langle {S^*Sx,x} \right\rangle } \right|^2  \right\}.
			\end{align*}
		But since
	\begin{align*}
	\left| {\left\langle {Sx,x} \right\rangle  } \right|    &\le \left\langle {\left| S \right|^{2\alpha}  x,x} \right\rangle ^{1/2} \left\langle {\left| {S^* } \right|^{2\left(1-\alpha\right)} x,x} \right\rangle ^{1/2}\qquad\qquad \text{(by \eqref{eq2.3})}
	\\
	&\le \left({ \frac{\left\langle {\left| S \right|^{2\alpha}  x,x} \right\rangle ^{r} + \left\langle {\left| {S^* } \right|^{2\left(1-\alpha\right)}  x,x} \right\rangle^{r} }{2} }\right)^{\frac{1}{r}}  \qquad  \text{(by \eqref{eq2.1})}
	\\
	&\le \left({ \frac{\left\langle {\left| S \right|^{2r\alpha}  x,x} \right\rangle  + \left\langle {\left| {S^* } \right|^{2r\left(1-\alpha\right)}  x,x} \right\rangle }{2} }\right)^{\frac{1}{r}}\qquad  \text{(by \eqref{eq2.2})} 
	\\
	&\le  \frac{1}{2^{\frac{1}{r}}} \left\langle {\left(\left| S \right|^{2r\alpha} +\left| {S^* } \right|^{2r\left(1-\alpha\right)}\right) x,x} \right\rangle     ^{\frac{1}{r}}
	\end{align*}
	it follows that 
		\begin{align}
\label{eq2.13}	\left| {\left\langle {S x,x} \right\rangle  } \right|^r   
	 \le  \frac{1}{2} \left\langle {\left(\left| S \right|^{2r\alpha} +\left| {S^* } \right|^{2r\left(1-\alpha\right)}\right) x,x} \right\rangle    
	\end{align}
	and
	\begin{align*}
	\left| {\left\langle {S^* Sx,x} \right\rangle } \right|  &\le \left\langle {\left| S^* S \right|^{2\alpha}  x,x} \right\rangle ^{1/2} \left\langle {\left| {S^*S } \right|^{2\left(1-\alpha\right)} x,x} \right\rangle ^{1/2}
	\\
	&\le \left({ \frac{\left\langle {\left| S^* S  \right|^{2\alpha}  x,x} \right\rangle ^{r} + \left\langle {\left| {S^* S  } \right|^{2\left(1-\alpha\right)}  x,x} \right\rangle^{r} }{2} }\right)^{\frac{1}{r}}  \qquad  \text{(by \eqref{eq2.1})}
	\\
	&\le \left({ \frac{\left\langle {\left|S^* S  \right|^{2r\alpha}  x,x} \right\rangle  + \left\langle {\left| {S^* S  } \right|^{2r\left(1-\alpha\right)}  x,x} \right\rangle }{2} }\right)^{\frac{1}{r}}\qquad  \text{(by \eqref{eq2.2})}
	\\
	&\le  \frac{1}{2^{\frac{1}{r}}} \left\langle {\left(\left| S^*S \right|^{2r\alpha} +\left| {S^*S } \right|^{2r\left(1-\alpha\right)}\right) x,x} \right\rangle     ^{\frac{1}{r}},
	\end{align*}
	it follows that 
	\begin{align}
\label{eq2.14}	\left| {\left\langle {S^*Sx,x} \right\rangle } \right|^r   \le  \frac{1}{2} \left\langle {\left(\left| S^* S \right|^{2r\alpha} +\left| {S^*S } \right|^{2r\left(1-\alpha\right)}\right) x,x} \right\rangle.     
	\end{align}
	Adding  \eqref{eq2.13} and \eqref{eq2.14}, we get
	\begin{align*}
	 & \frac{1}{2} \mathop {\sup }\limits_{\scriptstyle x \in \mathscr{H} \hfill \atop 
	 	\scriptstyle \left\| x \right\| = 1 \hfill}\left\langle {\left(\left| S \right|^{2r\alpha} +\left| {S^* } \right|^{2r\left(1-\alpha\right)}+\left| S^* S \right|^{2r\alpha} +\left| {S^* S } \right|^{2r\left(1-\alpha\right)}\right) x,x} \right\rangle     
	    \\
		&\ge\mathop {\sup }\limits_{\scriptstyle x \in \mathscr{H} \hfill \atop 
			\scriptstyle \left\| x \right\| = 1 \hfill} \left\{\left| {\left\langle {S x,x} \right\rangle  } \right|^r +		\left| {\left\langle {S^*S   x,x} \right\rangle  } \right|^r\right\}
		\\
		&= \mathop {\sup }\limits_{\scriptstyle x \in \mathscr{H} \hfill \atop 
			\scriptstyle \left\| x \right\| = 1 \hfill} \left\{\left(\left| {\left\langle {S x,x} \right\rangle  } \right|^2\right)^{r/2} +	\left(\left| {\left\langle {S^*S   x,x} \right\rangle  } \right|^2\right)^{r/2}\right\}
		\\
		&\ge 	\frac{1}{2^{\frac{r}{2} -1}}\mathop {\sup }\limits_{\scriptstyle x \in \mathscr{H} \hfill \atop 
			\scriptstyle \left\| x \right\| = 1 \hfill}	\left(\left| {\left\langle {S x,x} \right\rangle  } \right|^2  + \left| {\left\langle {S^*S   x,x} \right\rangle  } \right|^2\right)^{r/2}
	 = 	\frac{1}{2^{\frac{r}{2} -1}}	dw^r\left(S\right).
	\end{align*}
		Hence, 
	\begin{align*}
		 	dw^r\left(S\right) \le  \frac{2^{\frac{r}{2}}}{4}  \left\|\left| S \right|^{2r\alpha} +\left| {S^* } \right|^{2r\left(1-\alpha\right)}+\left| S^* S \right|^{2r\alpha} +\left| {S^* S } \right|^{2r\left(1-\alpha\right)}\right\|,
	\end{align*}
   as required.
\end{proof}
\begin{remark}
We note that, a refinement of the inequality \eqref{eq2.10} could deduced from \eqref{eq2.12}. Note that, by setting $r=2$ and $\alpha=\frac{1}{2}$ in \eqref{eq2.12}, we get \eqref{eq2.10}. Use the same proof given in Theorem \ref{thm5}, we can get
\begin{align*}
dw\left(S\right)  \le  \sqrt{\frac{1}{2}  w\left(\left| S \right|^{2} +\left| {S^* } \right|^{2}+2\left| S^*S \right|^{2} \right)}.
\end{align*}
Moreover, employing \eqref{eq2.8} for the above inequality we get
\begin{align*}
dw\left(S\right)   \le  \sqrt{\frac{1}{2}  w\left(\left| S \right|^{2} +\left| {S^* } \right|^{2}+2\left| S^*S \right|^{2} \right)} 
 \le \sqrt[4]{ \frac{1}{{8}} \left\|T^*T+TT^*\right\|},
\end{align*}
where $T=\left| S \right|^{2} +\left| {S^* } \right|^{2}+2\left| S^*S \right|^{2}$.  
\end{remark}

\begin{theorem}
\label{thm6}	Let    $S\in \mathscr{B} \left(\mathscr{H}\right)$,  $0 \le \alpha \le 1$ and $r\ge1$. Then
\begin{align}
dw^{2r}\left(S\right) \le 2^{r-1}   \left\|\alpha\left| S \right|^{2r}  +\left(1-\alpha\right)\left| {S^* } \right|^{2r} + \left| {S^*S} \right|^{2r}   \right\|.\label{eq2.15}
\end{align}
\end{theorem}
\begin{proof}
	Let $x\in \mathscr{H}$ be unit vector, then 
	\begin{align*}
	\left| {\left\langle {S x,x} \right\rangle  } \right|^2    &\le \left\langle {\left| S \right|^{2\alpha}  x,x} \right\rangle   \left\langle {\left| {S^* } \right|^{2\left(1-\alpha\right)}  x,x} \right\rangle   \qquad\qquad\qquad\qquad \text{(by \eqref{eq2.3})}
	\\
	&\le \left\langle {\left| S \right|^{2}  x,x} \right\rangle^{\alpha}  \left\langle {\left| {S^* } \right|^{2} x,x} \right\rangle^{\left(1-\alpha\right)} \qquad \qquad\qquad\qquad\text{(by \eqref{eq2.2})}
	\\
	&\le  \left(\alpha\left\langle {\left| S \right|^{2} x,x} \right\rangle^r  +\left(1-\alpha\right)\left\langle {\left| {S^* } \right|^{2}  x,x} \right\rangle^r \right)^{1/r}\qquad \text{(by \eqref{eq2.1})}
	\\
	&\le  \left(\alpha\left\langle {\left| S \right|^{2r} x,x} \right\rangle +\left(1-\alpha\right)\left\langle {\left| {S^* } \right|^{2r} x,x} \right\rangle  \right)^{1/r}\qquad \text{(by \eqref{eq2.3})}
	\\
	&\le   \left\langle {\left(\alpha\left| S\right|^{2r}  +\left(1-\alpha\right)\left| {S^* } \right|^{2r}  \right) x,x} \right\rangle  ^{1/r}.
	\end{align*}
	Therefore,
	\begin{align}
\label{eq2.16}	\left| {\left\langle {S x,x} \right\rangle  } \right|^{2r} \le \left\langle {\left(\alpha\left| S \right|^{2r}  +\left(1-\alpha\right)\left| {S^* } \right|^{2r}  \right) x,x} \right\rangle.     
	\end{align}
	Also, since $S^*S$ is selfadjoint then we have
		\begin{align}
\label{eq2.17}	\left| {\left\langle {S^*S   x,x} \right\rangle  } \right|^{2r} \le \left\langle {\left| {S^*S} \right|^{2r} x,x} \right\rangle.     
	\end{align}
		Adding \eqref{eq2.16} and \eqref{eq2.17}, we get
	\begin{align*}
  \mathop {\sup }\limits_{\scriptstyle x \in \mathscr{H} \hfill \atop 
		\scriptstyle \left\| x \right\| = 1 \hfill}\left\langle {\left(\alpha\left| S \right|^{2r}  +\left(1-\alpha\right)\left| {S^* } \right|^{2r} + \left| {S^*S} \right|^{2r}   \right) x,x} \right\rangle     
	&\ge\mathop {\sup }\limits_{\scriptstyle x \in \mathscr{H} \hfill \atop 
		\scriptstyle \left\| x \right\| = 1 \hfill} \left\{\left| {\left\langle {Sx,x} \right\rangle  } \right|^{2r} +		\left| {\left\langle {S^*S   x,x} \right\rangle  } \right|^{2r}\right\}
	\\
	&\ge 	\frac{2}{2^{r}}\mathop {\sup }\limits_{\scriptstyle x \in \mathscr{H} \hfill \atop 
		\scriptstyle \left\| x \right\| = 1 \hfill}	\left(\left| {\left\langle {S x,x} \right\rangle  } \right|^2  + \left| {\left\langle {S^*S   x,x} \right\rangle  } \right|^2\right)^{r }
	\\
	&= 	\frac{2}{2^{r}}	dw^{2r}\left(S\right).
	\end{align*}
	Hence,
\begin{align*}
dw^{2r}\left(S\right) \le 2^{r-1}   \left\|\alpha\left| S \right|^{2r}  +\left(1-\alpha\right)\left| {S^* } \right|^{2r} + \left| {S^*S} \right|^{2r}   \right\|.
\end{align*}
This completes the proof of Theorem \ref{thm6}.
\end{proof}

\begin{example}
Let $S = \left[ {\begin{array}{*{20}c}
	0 & 2  \\
	0& 0  \\
	\end{array}} \right]$. We have  $\|S\|=w\left(S\right)=2$. The upper bound of \eqref{eq1.1} gives $dw\left(S\right) \le 2\sqrt{5}=4.4721$. However, by applying \eqref{eq2.15} with $r=1$ and $\alpha=\frac{1}{2}$, we have $dw\left(S\right) \le 3\sqrt{2}=4.2426$, which implies that, the upper bound in \eqref{eq2.15} is better than the upper bound in \eqref{eq1.1}. 
\end{example}
 
 \begin{remark}
A refinement of the inequality \eqref{eq2.15} could deduced from the proof given in Theorem \ref{thm6}, we can get
 	\begin{align*}
 	dw^{2r}\left(S\right)  \le   2^{r-1}   w\left(\alpha\left| S \right|^{2r}  +\left(1-\alpha\right)\left| {S^* } \right|^{2r} + \left| {S^*S} \right|^{2r} \right).
 	\end{align*}
 	Moreover, employing \eqref{eq2.8} in the above inequality we get
 	\begin{align*}
 	dw^{2r}\left(S\right)  &\le     2^{r-1}   w\left(\alpha\left| S \right|^{2r}  +\left(1-\alpha\right)\left| {S^* } \right|^{2r} + \left| {S^*S} \right|^{2r} \right)
 	\\
 	&\le  2^{\frac{r}{2}-1} \left\|T_{r,\alpha}^*T_{r,\alpha}+T_{r,\alpha}T_{r,\alpha}^*\right\|^{1/2},
 	\end{align*}
 	where $T_{r,\alpha}=\alpha\left| S \right|^{2r}  +\left(1-\alpha\right)\left| {S^* } \right|^{2r} + \left| {S^*S} \right|^{2r}$. For $r=1$ and $\alpha=\frac{1}{2}$ the last inequality reduces to 
 	\begin{align*}
 dw\left(S\right)   \le  \sqrt{   w\left( \frac{1}{2}\left| S \right|^{2}  +\frac{1}{2}\left| {S^* } \right|^{2} + \left| {S^*S} \right|^{2}\right)}
 \le  \sqrt[4]{\frac{1}{2}\left\|T_{1,\frac{1}{2}}^*T_{1,\frac{1}{2}}+T_{1,\frac{1}{2}}T_{1,\frac{1}{2}}^*\right\| },
 \end{align*}
where $T_{1,\frac{1}{2}}=\frac{1}{2}\left| S \right|^{2}  +\frac{1}{2}\left| {S^* } \right|^{2} + \left| {S^*S} \right|^{2}$. 
 \end{remark}

\section{The Davis--Wielandt radius inequalities  for $n \times n$ matrix Operators}\label{sec3}
Several numerical radius type inequalities improving and refining
the inequality   
  \begin{align*}
\frac{1}{2}\left\|S\right\|\le w\left(S\right) \le \left\|S\right\| \qquad\qquad (S\in \mathscr{B}\left( \mathscr{H}\right))
\end{align*}
have been recently obtained by many
other authors see for example \cite{AF1}--\cite{BF},  and \cite{HD}. Among others, three important
facts concerning the  numerical radius inequalities   of $n \times
n$  operator matrices are obtained by different authors which are
grouped together, as follows:\\

Let  ${\bf S}=\left[S_{ij}\right]\in \mathscr{B}\left(\bigoplus _{i = 1}^n \mathscr{H}_i\right)$ such that $S_{ij}\in\mathscr{B}\left(\mathscr{H}_j, \mathscr{H}_i\right)$. Then
\begin{align}
\label{eq3.1} w\left(S\right)\le
\left\{ \begin{array}{l}
w \left( {\left[ {t_{ij}^{\left( 1 \right)} } \right]} \right),\qquad {\rm{Hou \,\&\, Du \,\,in}\,\,}\text{\cite{HD}}  
\\
w \left( {\left[ {t_{ij}^{\left( 2 \right)} } \right]} \right) ,\qquad {\rm{BaniDomi \,\&\, Kittaneh \,\,in}\,\,} \text{\cite{BF}} 
\\
w \left( {\left[ {t_{ij}^{\left( 3 \right)} } \right]} \right),\qquad  {\rm{AbuOmar \,\&\, Kittaneh \,\,in}\,\,} \text{\cite{AF1}}
\end{array} \right.;
\end{align}
where
\begin{align*}
t_{ij}^{\left( 1 \right)}  = w \left( {\left[ {\left\| {S_{ij} } \right\|} \right]} \right),
\,\,
t_{ij}^{\left( 2 \right)}  = \left\{ \begin{array}{l}
\frac{1}{2}\left( {\left\| {S_{ii} } \right\| + \left\| {S_{ii}^2 } \right\|^{1/2} } \right),\,\,\,\,\,\,\,i = j
\\
\left\| {S_{ij} } \right\|,\qquad\qquad\qquad\,\,\,\,\,\,\,\,\,\,i \ne j
\end{array} \right. , \,\,\,\,t_{ij}^{\left( 3 \right)}  = \left\{ \begin{array}{l}
w \left( {S_{ii} } \right),\,\,\,\,\,\,\,i = j
\\
\left\| {S_{ij} } \right\|,\,\,\,\,\,\,\,\,\,\, i \ne j
\end{array} \right..
\end{align*}

As mentioned in \cite{BP}, in our recent work \cite{Alomari1}  we tried to refine the last bound (above) proved by Abu Omar and  Kittaneh in \cite{AF1}; however there is a mistake in the printed version of the result. In the following result we correct \cite[Theorem 4.1]{Alomari1}.

 \begin{theorem}
\label{thm7} Let  ${\bf{S}}=\left[S_{ij}\right]\in
 	\mathscr{B}\left(\bigoplus _{i = 1}^n \mathscr{H}_i\right)$ such
 	that $S_{ij}\in\mathscr{B}\left(\mathscr{H}_j,
 	\mathscr{H}_i\right)$. Then
 	\begin{align}
 	w\left({\bf{S}}\right)\le  w\left( \left[s_{ij}\right] \right),
 	\end{align}
 	where
 	\begin{align*}
 	s_{ij}=  \left\{ \begin{array}{l}
 	w\left( {S_{ij} } \right), \qquad\qquad\qquad \,\,\,\,\,\,\,\,\,\, j = i \,\,\,\,\,{\rm{and}}\,\,\,\,\,j \ne k_i
 	\\
 	 w^{\frac{1}{2}} \left( {\left| {S_{ik_i } } \right|} \right)w^{\frac{1}{2}} \left( {\left| {S_{ik_i }^* } \right|} \right) ,\,\,\,\,\,\,\,\,\,\,  j = k_i
 	\,\,\,\,\,{\rm{and}}\,\,\,\,\,j \ne i
 	\\
 	\left\| {S_{ij} } \right\|, \qquad\qquad\qquad\,\,\,\,\,\,\,\,\,\,\,\,\, j \ne k_i \,\,\,\,{\rm{and}}\,\,\,\,\,j \ne i
 	\end{array} \right..
 	\end{align*}
 	where $k_i=n-i+1$.
 \end{theorem}

 \begin{proof}
 	Let ${\bf{x}} = \left[ {\begin{array}{*{20}c}
 		{x_1 } & {x_2 } &  \cdots  & {x_n }  \\
 		\end{array}} \right]^T \in \bigoplus _{i = 1}^n \mathscr{H}_i$ with $\|x\|=1$.  For simplicity setting $k_i=n-i+1$, then we have
 	\begin{align*}
 	\left| {\left\langle {{\bf{S}}{\bf{x}},{\bf{x}}} \right\rangle } \right| &= \left| {\sum\limits_{i,j = 1}^n {\left\langle {S_{ij} x_j ,x_i } \right\rangle } } \right| \\
 	&\le \sum\limits_{i,j = 1}^n {\left| {\left\langle {S_{ij} x_j ,x_i } \right\rangle } \right|}  \nonumber
 	\\
 	&\le \sum\limits_{i  = 1}^n {\left| {\left\langle {S_{ii} x_i ,x_i } \right\rangle } \right|} + \sum\limits_{i=1}^n {\left| {\left\langle {S_{ik_i } x_{k_i} ,x_{i} } \right\rangle } \right|}+ \sum\limits_{j \ne i,k_i }^n {\left| {\left\langle {S_{ij} x_j ,x_i } \right\rangle } \right|} \nonumber
 	\\
 	&\le \sum\limits_{i  = 1}^n {\left| {\left\langle {S_{ii} x_i ,x_i } \right\rangle } \right|} +  
 	\sum\limits_{i = 1,i \ne k_i }^n {\left\langle {\left| {S_{ik_i } } \right|x_{k_i } ,x_{k_i } } \right\rangle ^{\frac{1}{2}} \left\langle {\left| {S_{ik_i }^* } \right|x_i ,x_i } \right\rangle ^{\frac{1}{2}} } + \sum\limits_{j \ne i,k_i }^n {\left| {\left\langle {S_{ij} x_j ,x_i } \right\rangle } \right|} \nonumber
 	\\
 	&\le  \sum\limits_{i = 1}^n {   w \left( {S_{ii} } \right)\left\| {x_i } \right\|^2 } +
 	\sum\limits_{i = 1,i \ne k_i }^n {w^{\frac{1}{2}} \left( {\left| {S_{ik_i } } \right|} \right)w^{\frac{1}{2}} \left( {\left| {S_{ik_i }^* } \right|} \right)\left\| {x_{k_i } } \right\|\left\| {x_i } \right\|}  +\sum\limits_{j \ne i}^n {\left\| {S_{ij} } \right\|\left\| {x_i } \right\|\left\| {x_j } \right\|}
 	\\
 	&\le \sum_{i,j=1}^n {s_{ij}\left\| {x_i } \right\|\left\| {x_j } \right\|}
 	\\
 	&= \left\langle {\left[ {s_{ij} } \right]y,y} \right\rangle,
 	\end{align*}
 	where $y=\left( {\begin{array}{*{20}c}{\left\| {x_1 } \right\|} & {\left\| {x_2 } \right\|} &  \cdots  & {\left\| {x_n } \right\|}  \\  \end{array}} \right)^T$. Taking the supremum over ${\bf{x}}  \in \bigoplus \mathscr{H}_i$, we obtain the desired result.
 \end{proof}

In the next result, we present  Davis--Wielandt radius inequality  for $n \times n$ matrix Operators.
 

\begin{theorem}
	\label{thm8}Let  ${\bf T}=\left[T_{ij}\right]\in
	\mathscr{B}\left(\bigoplus _{i = 1}^n \mathscr{H}_i\right)$ such
	that $T_{ij}\in\mathscr{B}\left(\mathscr{H}_j,
	\mathscr{H}_i\right)$. Then
	\begin{align}
	 dw \left({\bf T}\right)\le  w \left( \left[t_{ij}\right] \right),\label{eq3.3}
	\end{align}
	where
	\begin{align*}
	t_{ij}=   \left\{ \begin{array}{l}
	w  \left( {T_{ii} } \right)  +\left\| {T_{ii} } \right\|^2    , \,\,\,\,\,\,\,\,\,\, j = i  
	\\
	\\
	\left\| {T_{ij} } \right\|  +\left\| {T_{ij} } \right\|^2   ,\,\,\,\,\,\,\,\,\,\,\,\,\, j \ne i 
	\end{array} \right..
	\end{align*}
\end{theorem}

\begin{proof}
	Let $x = \left[ {\begin{array}{*{20}c}
		{x_1 } & {x_2 } &  \cdots  & {x_n }  \\
		\end{array}} \right]^T \in \bigoplus _{i = 1}^n \mathscr{H}_i$ with $\|x\|=1$. 
	Then we have
	\begin{align*}
	dw\left({\bf T}\right)&=\mathop {\sup }\limits_{\scriptstyle {\bf x} \in \mathscr{H} \hfill \atop 
		\scriptstyle \left\| {\bf x} \right\| = 1 \hfill}\sqrt{\left| {\left\langle {{\bf T}{\bf x},{\bf x}} \right\rangle } \right|^2+\left| {\left\langle {{\bf T}^*{\bf T}{\bf x},{\bf x}} \right\rangle } \right|^2} 
	\\
	&\le\mathop {\sup }\limits_{\scriptstyle {\bf x} \in \mathscr{H} \hfill \atop 
		\scriptstyle \left\| {\bf x} \right\| = 1 \hfill}    \left\{\left| {\left\langle {{\bf T}{\bf x},{\bf x}} \right\rangle } \right| +\left| {\left\langle {{\bf T}^*{\bf T}{\bf x},{\bf x}} \right\rangle } \right| \right\} \qquad\qquad \text{(since $\sqrt{a+b}\le \sqrt{a}+\sqrt{b}$)}
	\end{align*}
	But since
	\begin{align}
	\left| {\left\langle {{\bf T}{\bf x},{\bf x}} \right\rangle } \right|   = \left| {\sum\limits_{i,j = 1}^n {\left\langle {T_{ij} x_j ,x_i } \right\rangle } } \right|  
	&\le \sum\limits_{i,j = 1}^n {\left| {\left\langle {T_{ij} x_j ,x_i } \right\rangle } \right| }  \nonumber
	\nonumber	\\
	&\le    \sum\limits_{i  = 1}^n {\left| {\left\langle {T_{ii} x_i ,x_i } \right\rangle } \right| }  +   \sum\limits_{j \ne i  }^n {\left| {\left\langle {T_{ij} x_j ,x_i } \right\rangle } \right| } \nonumber
	\nonumber\\
	&\le    \sum\limits_{i = 1}^n {   w  \left( {T_{ii} } \right)\left\| {x_i } \right\|^2  }  + \sum\limits_{j \ne i}^n {\left\| {T_{ij} } \right\| \left\| {x_i } \right\| \left\| {x_j } \right\| }\label{eq3.4}
	\end{align}
	where $y=\left( {\begin{array}{*{20}c}{\left\| {x_1 } \right\|} & {\left\| {x_2 } \right\|} &  \cdots  & {\left\| {x_n } \right\|}  \\  \end{array}} \right)^T$.

	Similarly, we have
	\begin{align}
	\left| {\left\langle {{\bf T}^*{\bf T}{\bf x},{\bf x}} \right\rangle } \right|  &= \left| {\sum\limits_{i,j = 1}^n {\left\langle {T_{ij}^* T_{ij}x_j,x_i } \right\rangle } } \right|  \nonumber
	\\
	&\le    \sum\limits_{i = 1}^n {   w  \left( {T^*_{ii}T_{ii} } \right)\left\| {x_i } \right\|^2  }  + \sum\limits_{j \ne i}^n {\left\| {T^*_{ij} T_{ij} } \right\| \left\| {x_i } \right\| \left\| {x_j } \right\| }.\label{eq3.5}
	\end{align}
	Adding \eqref{eq3.4} and \eqref{eq3.5}, we get
	\begin{align*}
	dw \left({\bf T}\right)&\le \mathop {\sup }\limits_{\scriptstyle {\bf x} \in \mathscr{H} \hfill \atop 
		\scriptstyle \left\| {\bf x} \right\| = 1 \hfill} \left\{\left| {\left\langle {{\bf T}{\bf x},{\bf x}} \right\rangle } \right| +\left| {\left\langle {{\bf T}^*{\bf T}{\bf x},{\bf x}} \right\rangle } \right| \right\}
	\\
	&\le  \sum\limits_{i = 1}^n { \left(  w  \left( {T_{ii} } \right)  +w  \left( {T^*_{ii}T_{ii} } \right) \right)\left\| {x_i } \right\|^2  }  +  \sum\limits_{j \ne i}^n {\left(\left\| {T_{ij} } \right\|  +\left\| {T^*_{ij} T_{ij} } \right\| \right)\left\| {x_i } \right\| \left\| {x_j } \right\| }
	\\
	&=  \sum\limits_{i = 1}^n { \left(  w  \left( {T_{ii} } \right)  +\left\| {T_{ii} } \right\|^2 \right)\left\| {x_i } \right\|^2  }  +  \sum\limits_{j \ne i}^n {\left(\left\| {T_{ij} } \right\|  +\left\| {T_{ij} } \right\|^2 \right)\left\| {x_i } \right\| \left\| {x_j } \right\| }
	\\
	&\le  \sum_{i,j=1}^n {t_{ij}\left\| {x_i } \right\|\left\| {x_j } \right\|}
	\\
	&=  \left\langle {\left[ {t_{ij} } \right]y,y} \right\rangle. 
	\end{align*}		
	Taking the supremum over ${\bf x} \in \bigoplus \mathscr{H}_i$, we obtain the right-hand side inequality in \eqref{eq3.3}, and this completes the proof.	
\end{proof}

\begin{corollary}
\label{cor2}	Let ${\bf{T}}={\left[ {\begin{array}{*{20}c}
			{T_{11} } & {T_{12} }  \\
			{T_{21} } & {T_{22} }  \\
			\end{array}} \right]} \in \mathscr{B}\left(\mathscr{H}_1\oplus\mathscr{H}_2\right)$. Then	
	\begin{align}
 \label{eq3.6} dw  \left( {{\bf{T}}} \right) \le   \frac{1}{2}\left( a + d + \sqrt {\left( {a - d} \right)^2  +\left(b+c\right)^2 }  \right),
	\end{align}
	where,
\begin{align*}
 a  = w \left( {T_{11} } \right) +\left\| {T_{11} } \right\|^2,\,\,
b  = \left\| {T_{12} } \right\| +\left\| {T_{12} } \right\|^2,
\,\,  
c = \left\| {T_{21} } \right\| +\left\| {T_{21} } \right\|^2,
\,\,
d   = w \left( {T_{22} } \right) +  \left\| {T_{22} } \right\|^2.\nonumber
\end{align*}
\end{corollary}	
\begin{proof}
Take $n=2$ in Theorem \ref{thm8}.	Let $a,b,c,d$ be as defined above. Then
	\begin{align*}
	dw  \left( {\left[ {\begin{array}{*{20}c}
			{T_{11} } & {T_{12} }  \\
			{T_{21} } & {T_{22} }  \\
			\end{array}} \right]} \right) &\le  w\left( {\left[ {\begin{array}{*{20}c}
			a & b  \\
			c & d  \\
			\end{array}} \right]} \right) \\ 
	&=  r\left( {\left[ {\begin{array}{*{20}c}
			a & {\frac{{b + c}}{2}}  \\
			{\frac{{b + c}}{2}} & d  \\
			\end{array}} \right]} \right) \\ 
	&=  \frac{1}{2}\left( a + d + \sqrt {\left( {a - d} \right)^2  +\left(b+c\right)^2 }  \right).
	\end{align*}
as required.	
\end{proof}	
\begin{corollary}
\label{cor3}	Let ${\left[ {\begin{array}{*{20}c}
			{T_{11} } & {0}  \\
			{0} & {T_{22} }  \\
			\end{array}} \right]} \in \mathscr{B}\left(\mathscr{H}_1\oplus\mathscr{H}_2\right)$, then
	\begin{align}
	dw  \left( {\left[ {\begin{array}{*{20}c}
			{T_{11} } & {0}  \\
			{0} & {T_{22} }  \\
			\end{array}} \right]} \right) \le  \max\left\{ w \left( {T_{11} } \right) +  \left\|T_{11}\right\|^2  , w  \left( {T_{22} } \right) +  \left\|T_{22}\right\|^2  \right\} 
	\end{align}
	In  special case, if $\mathscr{H}_1=\mathscr{H}_2$ and $ T_{11}=T_{22}=T$, then	
	\begin{align}
	dw  \left( {\left[ {\begin{array}{*{20}c}
			{T  } & {0}  \\
			{0} & {T  }  \\
			\end{array}} \right]} \right) \le  w  \left( {T} \right) +  \left\|T\right\|^2  
	\end{align}
\end{corollary}	
\begin{proof}	
From Corollary \ref{cor2}, we have	
	\begin{align*}
	dw  \left( {\left[ {\begin{array}{*{20}c}
			{T_{11} } & {0}  \\
			{0} & {T_{22} }  \\
			\end{array}} \right]} \right) &\le \max\left\{w \left( {T_{11} } \right) + w \left( {T_{11}^* T_{11} } \right),w  \left( {T_{22} } \right) + w  \left( {T_{22}^* T_{22} } \right)\right\}
	\\
	&=\max\left\{w  \left( {T_{11} } \right) + w  \left( { \left|T_{11}\right|^2 } \right),w  \left( {T_{22} } \right) + w  \left( {\left|T_{22}\right|^2 } \right)\right\}
	\\
	&\le\max\left\{w  \left( {T_{11} } \right) +  \left\|T_{11}\right\|^2 ,w  \left( {T_{22} } \right) +  \left\|T_{22}\right\|^2 \right\},
	\end{align*}	
as required.	
\end{proof}

\begin{corollary}
	Let ${\bf{T}}={\left[ {\begin{array}{*{20}c}
			{T } & {S}  \\
			{S } & {T}  \\
			\end{array}} \right]} \in \mathscr{B}\left(\mathscr{H} \oplus\mathscr{H} \right)$. Then	
	\begin{align}
	dw  \left( {{\bf{T}}} \right) \le    w  \left( {T} \right) +  \left\| {T } \right\|^2+ \left\| {S} \right\|   + \left\| {S } \right\|^2
	\end{align}

\end{corollary}

\begin{proof}
From Corollary \ref{cor2}, we have $T_{11}=T_{22}=T$ and $T_{12}=T_{21}=S$, therefore
\begin{align*}
a  = w \left( {T} \right) +\left\| {T} \right\|^2=d,\qquad
b  = \left\| {S} \right\| +\left\| {S} \right\|^2=c.
\end{align*}
Thus,
\begin{align*}
dw  \left( {\left[ {\begin{array}{*{20}c}
		{T } & {S}  \\
		{S } & {T}  \\
		\end{array}} \right]} \right)\le  a+b = w \left( {T} \right) +\left\| {T} \right\|^2+ \left\| {S} \right\| +\left\| {S} \right\|^2,
\end{align*}
as required.
\end{proof}


A refinement of Theorem \ref{thm8} is formulated as follows:

\begin{theorem}
	\label{thm9}Let  ${\bf T}=\left[T_{ij}\right]\in
	\mathscr{B}\left(\bigoplus _{i = 1}^n \mathscr{H}_i\right)$ such
	that $T_{ij}\in\mathscr{B}\left(\mathscr{H}_j,
	\mathscr{H}_i\right)$. Then
	\begin{align}
	\label{eq3.10}	\frac{1}{\sqrt{2}}\left\|{{\bf{T}} + {\bf{T}}^* {\bf{T}}} \right\|\le dw \left({\bf T}\right)\le  w^{1/2} \left( \left[t_{ij}\right] \right),
	\end{align}
	where
	\begin{align*}
	t_{ij}=  n\cdot\left\{ \begin{array}{l}
	w^2 \left( {T_{ii} } \right)  + \left\| {T_{ii} } \right\|^4    , \,\,\,\,\,\,\,\,\,\, j = i  
	\\
	\\
	\left\| {T_{ij} } \right\|^2 +\left\| {T_{ij} } \right\|^4  ,\,\,\,\,\,\,\,\,\,\,\,\,\, j \ne i 
	\end{array} \right..
	\end{align*}
\end{theorem}

\begin{proof}
	Let ${\bf x} = \left[ {\begin{array}{*{20}c}
		{x_1 } & {x_2 } &  \cdots  & {x_n }  \\
		\end{array}} \right]^T \in \bigoplus _{i = 1}^n \mathscr{H}_i$ with $\|{\bf x}\|=
	\sum\limits_{i = 1}^n {\left\| {x_i } \right\|^2 }  = 1$. 
	Then we have
	\begin{align*}
	dw\left({\bf T}\right)=\mathop {\sup }\limits_{\scriptstyle {\bf x} \in \mathscr{H} \hfill \atop 
		\scriptstyle \left\| {\bf x} \right\| = 1 \hfill} \left\{ \sqrt{
		\left| {\left\langle {{\bf T}{\bf x},{\bf x}} \right\rangle } \right|^2  + \left\| {{\bf T}{\bf x}} \right\|^4 }\right\}
	=\mathop {\sup }\limits_{\scriptstyle {\bf x} \in \mathscr{H} \hfill \atop 
		\scriptstyle \left\| {\bf x} \right\| = 1 \hfill}\sqrt{\left| {\left\langle {{\bf T}{\bf x},{\bf x}} \right\rangle } \right|^2+\left| {\left\langle {{\bf T}^*{\bf T}{\bf x},{\bf x}} \right\rangle } \right|^2}.
	\end{align*}
	But since
	\begin{align}
	\left| {\left\langle {{\bf T}{\bf x},{\bf x}} \right\rangle } \right|^2 &= \left| {\sum\limits_{i,j = 1}^n {\left\langle {T_{ij} x_j ,x_i } \right\rangle } } \right|^2 
	\nonumber\\
	&\le n\cdot\sum\limits_{i,j = 1}^n {\left| {\left\langle {T_{ij} x_j ,x_i } \right\rangle } \right|^2}  \nonumber \qquad\qquad \text{(by Jensen's inequality)}
	\nonumber	\\
	&\le  n\cdot \sum\limits_{i  = 1}^n {\left| {\left\langle {T_{ii} x_i ,x_i } \right\rangle } \right|^2}  + n\cdot \sum\limits_{j \ne i  }^n {\left| {\left\langle {T_{ij} x_j ,x_i } \right\rangle } \right|^2} \nonumber
	\nonumber\\
	&\le   n\cdot\sum\limits_{i = 1}^n {   w^2 \left( {T_{ii} } \right)\left\| {x_i } \right\|^ 4 }  + n\cdot\sum\limits_{j \ne i}^n {\left\| {T_{ij} } \right\|^2\left\| {x_i } \right\|^2\left\| {x_j } \right\|^2}\nonumber
	\\
	&\le   n\cdot\sum\limits_{i = 1}^n {   w^2 \left( {T_{ii} } \right)\left\| {x_i } \right\|^2  }  + n\cdot\sum\limits_{j \ne i}^n {\left\| {T_{ij} } \right\|^2\left\| {x_i } \right\| \left\| {x_j } \right\|  }, \label{eq3.11}	
	\end{align}
	the last inequality holds, since $\|x_i\|^4 \le \|x_i\|^2 \le1$ and $\|x_i\|^2 \le \|x_i\|\le1$ for all $i=,\cdots, n$;	where $y=\left( {\begin{array}{*{20}c}{\left\| {x_1 } \right\|} & {\left\| {x_2 } \right\|} &  \cdots  & {\left\| {x_n } \right\|}  \\  \end{array}} \right)^T$. \\

	Similarly, we have
	\begin{align}
	\left| {\left\langle {{\bf T}^*{\bf T}{\bf x},{\bf x}} \right\rangle } \right|^2 &= \left| {\sum\limits_{i,j = 1}^n {\left\langle {T_{ij}^* T_{ij}x_j,x_i } \right\rangle } } \right|^2 \nonumber
	\\
	&\le   n\cdot\sum\limits_{i = 1}^n {   w^2 \left( {T^*_{ii}T_{ii} } \right)\left\| {x_i } \right\|^2 }  + n\cdot\sum\limits_{j \ne i}^n {\left\| {T^*_{ij} T_{ij} } \right\|^2\left\| {x_i } \right\| \left\| {x_j } \right\| }.\label{eq3.12}
	\end{align}
	Adding \eqref{eq3.11} and \eqref{eq3.12}, we get 
	\begin{align*}
	dw^2\left({\bf T}\right)&=\mathop {\sup }\limits_{\scriptstyle {\bf x} \in \mathscr{H} \hfill \atop 
		\scriptstyle \left\| {\bf x} \right\| = 1 \hfill} \left\{\left| {\left\langle {{\bf T}{\bf x},{\bf x}} \right\rangle } \right|^2+\left| {\left\langle {{\bf T}^*{\bf T}{\bf x},{\bf x}} \right\rangle } \right|^2 \right\}
	\\
	&\le n\cdot\sum\limits_{i = 1}^n { \left(  w^2 \left( {T_{ii} } \right)  +w^2 \left( {T^*_{ii}T_{ii} } \right) \right)\left\| {x_i } \right\|^2 }  + n\cdot\sum\limits_{j \ne i}^n {\left(\left\| {T_{ij} } \right\|^2 +\left\| {T^*_{ij} T_{ij} } \right\|^2\right)\left\| {x_i } \right\| \left\| {x_j } \right\| }
	\\
	&= n\cdot  \sum\limits_{i = 1}^n { \left(  w^2  \left( {T_{ii} } \right)  +\left\| {T_{ii} } \right\|^4 \right)\left\| {x_i } \right\|^2  }  +  \sum\limits_{j \ne i}^n {\left(\left\| {T_{ij} } \right\|^2  +\left\| {T_{ij} } \right\|^4 \right)\left\| {x_i } \right\|  \left\| {x_j } \right\|  }
	\\
	&\le  n\cdot\sum_{i,j=1}^n {t_{ij}\left\| {x_i } \right\| \left\| {x_j } \right\| }
	\\
	&=  n\cdot\left\langle {\left[ {t_{ij} } \right]y,y} \right\rangle. 
	\end{align*}		
	Taking the supremum over ${\bf x} \in \bigoplus \mathscr{H}_i$, we obtain the right-hand side inequality.
	
	To prove the left hand side inequality we note that
	\begin{align*}
	dw^2\left({\bf T}\right)&=\mathop {\sup }\limits_{\scriptstyle {\bf x} \in \mathscr{H} \hfill \atop 
		\scriptstyle \left\| {\bf x} \right\| = 1 \hfill} \left\{ \left| {\left\langle {{\bf T}{\bf x},{\bf x}} \right\rangle } \right|^2+\left| {\left\langle {{\bf T}^*{\bf T}{\bf x},{\bf x}} \right\rangle } \right|^2 \right\}
	\\
	&\ge\frac{1}{2}\mathop {\sup }\limits_{\scriptstyle {\bf x} \in \mathscr{H} \hfill \atop 
		\scriptstyle \left\| {\bf x} \right\| = 1 \hfill} \left\{ \left| {\left\langle {{\bf T}{\bf x},{\bf x}} \right\rangle } \right| +\left| {\left\langle {{\bf T}^*{\bf T}{\bf x},{\bf x}} \right\rangle } \right|  \right\}^2
	\\
	&\ge\frac{1}{2}\mathop {\sup }\limits_{\scriptstyle {\bf x} \in \mathscr{H} \hfill \atop 
		\scriptstyle \left\| {\bf x} \right\| = 1 \hfill} \left\{ \left|\left\langle {\left( {{\bf{T}} + {\bf{T}}^* {\bf{T}}} \right){\bf{x}},{\bf{x}}} \right\rangle \right|^2  \right\} 
	\\
	&=\frac{1}{2}\left\|{{\bf{T}} + {\bf{T}}^* {\bf{T}}} \right\|,
	\end{align*}		
	as required.	
\end{proof}

\begin{corollary}
	\label{cor5}Let ${\bf{T}}={\left[ {\begin{array}{*{20}c}
			{T_{11} } & {T_{12} }  \\
			{T_{21} } & {T_{22} }  \\
			\end{array}} \right]} \in \mathscr{B}\left(\mathscr{H}_1\oplus\mathscr{H}_2\right)$. Then	
	\begin{align}
	\label{eq3.13}  dw  \left( {{\bf{T}}} \right) \le  \sqrt{   a + d + \sqrt {\left( {a - d} \right)^2  +\left(b+c\right)^2 }},
	\end{align}
	where,
	\begin{align*}
	a  = w^2 \left( {T_{11} } \right) + \left\|{ T_{11} } \right\|^4,\,\,
	b  =\left\| {T_{12} } \right\|^2 +\left\| {T_{12} } \right\|^4,
	\,\,  
	c = \left\| {T_{21} } \right\|^2 +\left\| {T_{21} } \right\|^4,\,\,
	d  = w^2 \left( {T_{22} } \right) +   \left\| {T_{22} } \right\|^4.\nonumber
	\end{align*}
\end{corollary}	
\begin{proof}
	Take $n=2$ in Theorem \ref{thm9}. Let $a,b,c,d$ be as defined above. Then
	\begin{align*}
	dw^2 \left( {\left[ {\begin{array}{*{20}c}
			{T_{11} } & {T_{12} }  \\
			{T_{21} } & {T_{22} }  \\
			\end{array}} \right]} \right) &\le 2w\left( {\left[ {\begin{array}{*{20}c}
			a & b  \\
			c & d  \\
			\end{array}} \right]} \right) \\ 
	&= 2r\left( {\left[ {\begin{array}{*{20}c}
			a & {\frac{{b + c}}{2}}  \\
			{\frac{{b + c}}{2}} & d  \\
			\end{array}} \right]} \right) \\ 
	&= 2r\left( {\left[ {\begin{array}{*{20}c}
			a & {\frac{{b + c}}{2}}  \\
			{\frac{{b + c}}{2}} & d  \\
			\end{array}} \right]} \right) \\ 
	&=   a + d + \sqrt {\left( {a - d} \right)^2  + \left(b+c\right)^2 }. 
	\end{align*}
	which proves the required inequality.
\end{proof}	
\begin{corollary}
\label{cor6}	Let ${\left[ {\begin{array}{*{20}c}
			{T_{11} } & {0}  \\
			{0} & {T_{22} }  \\
			\end{array}} \right]} \in \mathscr{B}\left(\mathscr{H}_1\oplus\mathscr{H}_2\right)$, then
	\begin{align}
	dw  \left( {\left[ {\begin{array}{*{20}c}
			{T_{11} } & {0}  \\
			{0} & {T_{22} }  \\
			\end{array}} \right]} \right) \le \sqrt{2}\max\left\{\sqrt{w^2 \left( {T_{11} } \right) +  \left\|T_{11}\right\|^4} ,\sqrt{w^2 \left( {T_{22} } \right) +  \left\|T_{22}\right\|^4 }\right\} 
	\end{align}
	In  special case, if $\mathscr{H}_1=\mathscr{H}_2$ and $ T_{11}=T_{22}=T$, then	
	\begin{align}
	dw  \left( {\left[ {\begin{array}{*{20}c}
			{T  } & {0}  \\
			{0} & {T  }  \\
			\end{array}} \right]} \right) \le \sqrt{2}\left(w^2 \left( {T } \right) +  \left\|T \right\|^4 \right)^{1/2}.
	\end{align}
\end{corollary}	
\begin{proof}	
	Form Corollary \ref{cor5}, we have	
	\begin{align*}
	dw^2 \left( {\left[ {\begin{array}{*{20}c}
			{T_{11} } & {0}  \\
			{0} & {T_{22} }  \\
			\end{array}} \right]} \right) &\le2\max\left\{w^2 \left( {T_{11} } \right) + w^2 \left( {T_{11}^* T_{11} } \right),w^2 \left( {T_{22} } \right) + w^2 \left( {T_{22}^* T_{22} } \right)\right\}
	\\
	&=2\max\left\{w^2 \left( {T_{11} } \right) + w^2 \left( { \left|T_{11}\right|^2 } \right),w^2 \left( {T_{22} } \right) + w^2 \left( {\left|T_{22}\right|^2 } \right)\right\}
	\\
	&\le 2\max\left\{w^2 \left( {T_{11} } \right) +  \left\|T_{11}\right\|^4 ,w^2 \left( {T_{22} } \right) +  \left\|T_{22}\right\|^4 \right\},
	\end{align*}	
	which gives the desired result. 
\end{proof}	

\begin{remark}
Using the same approach considered in Theorem \ref{thm7}, one can refine Theorems \ref{thm8} and \ref{thm9}.
\end{remark} 

Finally, we introduce the concept of the Euclidean Davis--Wielandt radius. In fact, for an $n$-tuple $ {\bf{S}} = \left( {S_1 , \cdots ,S_n } \right) \in  \mathscr{B}\left(\mathscr{H}\right)^{n}:=\mathscr{B}\left(\mathscr{H}\right)\times \cdots \times \mathscr{B}\left(\mathscr{H}\right)$; i.e.,  for $S_1,\cdots,S_n \in \mathscr{B}\left(\mathscr{H}\right)$,
one of the most interesting generalization of the Davis--Wielandt radius $dw\left(  \cdot   \right)$, is the Euclidean Davis--Wielandt radius, which is defined  as:
\begin{align}
\label{eq3.16}dw_{\rm{e}}\left( S_1,\cdots,S_n \right) = \mathop {\sup }\limits_{\scriptstyle x \in H \hfill \atop 
	\scriptstyle \left\| x \right\| = 1 \hfill} \left( { \sum\limits_{i = 1}^n \left(\left| {\left\langle {S_ix,x} \right\rangle } \right|^2  + \left\| {S_ix} \right\|^4\right) } \right)^{1/2}.
\end{align}
Indeed, a nice relation between the Euclidean operator radius \eqref{eq2.6} and the Euclidean Davis--Wielandt radius \eqref{eq3.16}, can be constructed as follows:\\

For any positive integer $n$, let $T_i\in\mathscr{B}\left(\mathscr{H}\right)$ $(i=1,\cdots,2n)$. Therefore, we have
\begin{align*}
w_{\rm{e}} \left( {T_1 , \cdots ,T_{2n} } \right): = \mathop {\sup }\limits_{\left\| x \right\| = 1} \left( {\sum\limits_{i = 1}^{2n} {\left| {\left\langle {T_i x,x} \right\rangle } \right|^2 } } \right)^{1/2}\qquad\text{for all}\,\, x\in\mathscr{H}. 
\end{align*}
Let $S_i\in\mathscr{B}\left(\mathscr{H}\right)$ $(i=1,\cdots,n)$. Construct the following sequence of operators $S_i$ in terms of $T_i$, given as:
\begin{align*}
T_1&=S_1, \qquad\text{and} \qquad T_2=S_1^*S_1; 
\\
T_3&=S_2, \qquad\text{and} \qquad T_4=S_2^*S_2 ;
\\
T_5&=S_3, \qquad\text{and} \qquad T_6=S_3^*S_3 ;
\\
\vdots  
\\
T_{2n-1}&=S_n, \qquad\text{and} \qquad T_{2n}=S_n^*S_n.
\end{align*}
Now, we have
\begin{align*}
w_{\rm{e}} \left( {T_1 , \cdots ,T_{2n} } \right)&:= \mathop {\sup }\limits_{\left\| x \right\| = 1} \left( {\sum\limits_{i = 1}^{2n} {\left| {\left\langle {T_i x,x} \right\rangle } \right|^2 } } \right)^{1/2} 
\\
&= \mathop {\sup }\limits_{\left\| x \right\| = 1} \left( {\sum\limits_{i = 1}^{n} {\left(\left| {\left\langle {S_i x,x} \right\rangle } \right|^2 +\left| {\left\langle {S^*_iS_i x,x} \right\rangle } \right|^2\right)} } \right)^{1/2}
\\
&=dw_{\rm{e}}\left( S_1,\cdots,S_n \right).
\end{align*}
which gives a very elegant relation between the Euclidean operator radius and the Euclidean Davis--Wielandt radius.

Now, from the definition of the Euclidean Davis--Wielandt radius \eqref{eq3.16}, we have
\begin{align*}
dw_{\rm{e}}\left( S_1,\cdots,S_n \right)&=\mathop {\sup }\limits_{\scriptstyle x \in H \hfill \atop 
	\scriptstyle \left\| x \right\| = 1 \hfill} \left( { \sum\limits_{i = 1}^n\left| {\left\langle {S_ix,x} \right\rangle } \right|^2  + \left\| {S_ix} \right\|^4 } \right)^{1/2}
\\
&= \mathop {\sup }\limits_{\left\| x \right\| = 1} \left( {\sum\limits_{i = 1}^{n} {\left(\left| {\left\langle {S_i x,x} \right\rangle } \right|^2 +\left| {\left\langle {S^*_iS_i x,x} \right\rangle } \right|^2\right)} } \right)^{1/2} 
\\
&\le \left( {  \mathop {\sup }\limits_{\left\| x \right\| = 1} \sum\limits_{i = 1}^{n} \left| {\left\langle {S_i x,x} \right\rangle } \right|^2 + \mathop {\sup }\limits_{\left\| x \right\| = 1}\sum\limits_{i = 1}^{n}  \left| {\left\langle {S^*_iS_i x,x} \right\rangle } \right|^2  } \right)^{1/2} 
\\
&\le \left( {  \mathop {\sup }\limits_{\left\| x \right\| = 1} \sum\limits_{i = 1}^{n} \left| {\left\langle {S_i x,x} \right\rangle } \right|^2 } \right)^{1/2}+ \left( { \mathop {\sup }\limits_{\left\| x \right\| = 1}\sum\limits_{i = 1}^{n}  \left| {\left\langle {S^*_iS_i x,x} \right\rangle } \right|^2  } \right)^{1/2} 
\\
&= w_{\rm{e}}\left( S_1,\cdots,S_n \right)+w_{\rm{e}}\left( \left|S_1\right|^2,\cdots,\left|S_n\right|^2 \right).
\end{align*}
Also, one can observe that
\begin{align*}
dw_{\rm{e}}\left( S_1,\cdots,S_n \right)\ge \max\left\{w_{\rm{e}}\left( S_1,\cdots,S_n \right),w_{\rm{e}}\left( \left|S_1\right|^2,\cdots,\left|S_n\right|^2 \right)\right\}.
\end{align*}
Thus, we just proved the following result.
\begin{theorem}
	Let $S_i\in\mathscr{B}\left(\mathscr{H}\right)$ $(i=1,\cdots,n)$. Then,
	\begin{align*}
	\max\left\{w_{\rm{e}}\left( S_1,\cdots,S_n \right),w_{\rm{e}}\left( \left|S_1\right|^2,\cdots,\left|S_n\right|^2 \right)\right\} &\le dw_{\rm{e}}\left( S_1,\cdots,S_n \right) 
	\\
	&\le w_{\rm{e}}\left( S_1,\cdots,S_n \right)+w_{\rm{e}}\left( \left|S_1\right|^2,\cdots,\left|S_n\right|^2 \right).
	\end{align*}
\end{theorem}

One can generalizes the results in Section \ref{sec2}, by following the same procedure above. As a direct result, from Lemma \ref{lemma5} and Theorem \ref{thm3}, one can easily observe that
\begin{align*}
\frac{1}{4}\left\| {\sum\limits_{k = 1}^n { \left(\left|S_k\right|^2 +\left|S_k^*\right|^2 +2\left|S_k\right|^4  \right) } } \right\| \le dw^2_{\rm{e}}\left( S_1,\cdots,S_n \right)\le \frac{1}{2}\left\| {\sum\limits_{k = 1}^n {\left(\left|S_k\right|^2 +\left|S_k^*\right|^2 +2\left|S_k\right|^4  \right)    } } \right\|,
\end{align*}
by setting $p=1$ in \eqref{eq2.9}, taking into account the number of operators in \eqref{eq2.9} is $2n$ instead of $n$ and the previous mentioned sequence of operators. We leave the rest of other generalizations for the interested reader. 

\begin{remark}
	In Lemma \ref{lemma4}, we have shown that $w_{\rm{e}} \left( {S , S^*S} \right)= dw\left(S\right)$. Using the same idea, we generalize  the     Davis--Wielandt radius using the generalized Euclidean operator radius $w_p \left( {\cdot , \cdot} \right)$. Since we have
	\begin{align} 
	\label{eq3.17}w_{p} \left( {T_1 , \cdots ,T_n } \right): = \mathop {\sup }\limits_{\left\| x \right\| = 1} \left( {\sum\limits_{i = 1}^n {\left| {\left\langle {T_i x,x} \right\rangle } \right|^p } } \right)^{1/p}, \qquad  p\ge1.
	\end{align} 
	Therefore,
	by setting $n=2$, $T_1=S$ and $T_2=S^*S$\,\,\,  $\left(S\in \mathscr{B}\left(\mathscr{H}\right)\right)$ in \eqref{eq3.17}, we have
	\begin{align*}
	w_p \left( {S , S^*S} \right)&:= \mathop {\sup }\limits_{\left\| x \right\| = 1} \left( { \left| {\left\langle {S x,x} \right\rangle } \right|^p   +\left| {\left\langle {S^*S x,x} \right\rangle } \right|^p } \right)^{1/p}
	\\
	&= \mathop {\sup }\limits_{\left\| x \right\| = 1} \left\{ {\sqrt[p] {\left| {\left\langle {Sx,x} \right\rangle } \right|^p  + \left\| {Sx} \right\|^{2p} } } \right\}
	\\
	&= dw_p\left(S\right)
	\end{align*}
	for all  $p\ge1$, and this is called the generalized Euclidean  Davis--Wielandt radius of $S$. Clearly, for $p=2$ we refer to the well-known  Davis--Wielandt radius, $dw_2\left(S\right)=dw\left(S\right)$.
	
	As an immediate consequence of Theorem \ref{thm3}, one can easily observe that
	\begin{align*}
	\frac{1}{2^{ p}}\left\| { \frac{\left|S\right|^2  +\left|S^*\right|^2}{2} + \left|S\right|^4 } \right\|^p \le  dw^{2p}_{2p}\left(S\right) \le  \left\| { \left( {\frac{\left|S\right|^2  +\left|S^*\right|^2}{2} } \right)^p + \left|S\right|^{4p}     } \right\|,\qquad p\ge1.
	\end{align*}
	At the end, one can use the presented inequalities in  \cite{P}--\cite{SMS}, to obtain several bounds for $dw_p \left( {\cdot} \right)$.  
\end{remark}
 


\end{document}